\documentclass[12pt]{amsart}
\usepackage[left=3cm,top=3cm,right=3cm]{geometry}
\usepackage{amsmath}
\usepackage{amssymb}
\usepackage{amscd}
\usepackage[lined,algonl,boxed,norelsize]{algorithm2e}
\let\chapter\undefined
\usepackage{hyperref}

\numberwithin{equation}{section}

\newtheorem{theorem}[equation]{Theorem}
\newtheorem{lemma}[equation]{Lemma}

\newtheorem{proposition}[equation]{Proposition}
\newtheorem{corollary}[equation]{Corollary}
\newtheorem*{theorem*}{Theorem}

\theoremstyle{definition}
\newtheorem*{definition}{Definition}

\newtheorem{construction}[equation]{Construction}

\theoremstyle{remark}
\newtheorem{remark}[equation]{Remark}

\newcommand{\E}{{\mathcal E}}

\newcommand{\A}{{\mathcal A}}

\newcommand{\M}{{\mathcal M}}

\def\QQ{{\mathbb Q}}

\def\RR{{\mathbb R}}
\def\ZZ{{\mathbb Z}}

\newcommand{\Div}{\operatorname{Div}}
\newcommand{\divisor}{\operatorname{div}}
\newcommand{\Jac}{\operatorname{Jac}}
\newcommand{\Prin}{\operatorname{Prin}}

\newcommand{\Pic}{\operatorname{Pic}}

\newcommand{\Hom}{\operatorname{Hom}}

\newcommand{\outdeg}{{\operatorname{outdeg}}}

\newcommand{\diam}{\operatorname{diam}}
\newcommand{\CPA}{\operatorname{CPA}}
\newcommand{\BDV}{\operatorname{BDV}}
\newcommand{\Meas}{\operatorname{Meas}}

\begin{document}

\title{Chip-firing games, potential theory on graphs, and spanning trees}

\date{\today}

\subjclass[2010]{05C05, 05C25, 05C50, 05C57, 05C85}

\author{Matthew Baker}
\email{mbaker@math.gatech.edu}
\author{Farbod Shokrieh}
\email{shokrieh@math.gatech.edu}
\address{Georgia Institute of Technology\\
Atlanta, Georgia 30332-0160\\
USA}

\begin{abstract}
We study the interplay between chip-firing games and potential theory on graphs, characterizing 
reduced divisors ($G$-parking functions) on graphs as the solution to an energy (or potential) minimization problem and 
providing an algorithm to efficiently compute reduced divisors.  Applications include
an ``efficient bijective'' proof of Kirchhoff's matrix-tree theorem and a new algorithm for finding random spanning trees.
The running times of our algorithms are analyzed using potential theory, and we show that the bounds thus obtained 
generalize and improve upon several previous results in the literature. We also extend some of these considerations to metric graphs.

\end{abstract}

\thanks{We would like to thank Omid Amini, Jan Draisma, Xander Faber, Greg Lawler, Ye Luo, Sam Payne, and the anonymous referees for their comments on an earlier version of this manuscript. The authors' work was supported in part by NSF grants DMS-0901487 and CCF-0902717.}

\maketitle

\section{Introduction}
\label{IntroSection}

Chip-firing games on graphs arise in several different fields of research: in theoretical physics they relate to the ``abelian sandpile" or ``abelian avalanche" models in the context of self-organized critical phenomena (\cite{BTW88,Dhar90,Gabrielov93}); in arithmetic geometry, they appear implicitly in the study of component groups of N{\'e}ron models of Jacobians of algebraic curves (\cite{RaynaudPicard, Lorenzini89, BN1}); and in algebraic graph theory they relate to the study of flows and cuts in graphs (\cite{BacherHN97,Biggs97,Biggs99}).  We recommend the recent survey article 
\cite{WhatSandpile} for a short but more detailed overview of the subject.

\medskip

There is a close connection between chip-firing games and potential theory on graphs. In this paper, we explore some new aspects of this interplay. Conceptually, this connection should not come as a surprise; in both settings the Laplacian operator plays a crucial rule. However, in chip-firing games an extra ``integrality condition'' is imposed; in the language of optimization theory, chip-firing games lead to integer programing problems whose associated linear programming relaxations can be solved using potential theory on graphs. Our potential theory methods allow us to prove some new results about chip-firing games and to give new proofs
and/or generalizations of some known results in the subject. We also show that certain ``ad-hoc'' techniques in the literature are naturally explained or unified by our approach.

\medskip

Our main potential-theoretic tool is the {\em energy pairing} (see \S\ref{energySection}), which is a canonical positive definite bilinear form defined on the set of divisors of degree zero%
\footnote{A {\em divisor} on a finite graph $G$ is an element $\sum_{v \in V(G)} a_v (v)$ of the free abelian group $\Div(G)$ on the set $V(G)$ of vertices of $G$.
The {\em degree} of a divisor is the sum of the $a_v$ over all $v$.} 
on $G$. This pairing can be computed using any {\em generalized inverse} of the Laplacian matrix of $G$. The energy pairing can be used to define two functions $\E_q$ and $b_q$ (for a fixed vertex $q$) on $\Div(G)$ which interact in a useful way with {\em chip-firing moves}; 
after firing a set of vertices $A$ (not containing the exceptional vertex $q$), the value of $\E_q$ goes down by at least the size of the associated cut 
(see Proposition~\ref{energyproposition}), and the value of $b_q$ goes down by {\em exactly} the size of $A$ 
(see Proposition~\ref{energyproposition2}). 

\medskip

Chip-firing moves induce a natural equivalence relation on $\Div(G)$ called {\em linear equivalence} of divisors. 
If we once again fix a vertex $q$, then a particularly nice set of representatives for linear equivalence classes is given by the {\em $q$-reduced divisors} (see \S\ref{ReducedDivisorsSubsection}). We show that the $q$-reduced divisor equivalent to a given divisor $D$ can be characterized as the unique element of $|D|_q$ (the set of divisors $D' = \sum_v a'_v (v)$ linearly equivalent to $D$ for which $a'_v \geq 0$ whenever $v \neq q$) minimizing the functional $\E_q$; see Theorem~\ref{energyThm}..  A similar result holds with $\E_q$ replaced by $b_q$; see Theorem~\ref{energyThm2}.  
Using this result we are able to give a new proof of the important fact that there is a unique $q$-reduced divisor in each
linear equivalence class.

\medskip

In order to check whether or not a given divisor is $q$-reduced, according to the definition (given in \S\ref{ReducedDivisorsSubsection}), one needs to check a certain inequality for {\em all subsets} of $V(G) \backslash \{ q \}$.
But there is in fact a much more efficient procedure called {\em Dhar's burning algorithm} (after Dhar \cite{Dhar90}); 
see \S\ref{DharBurningAlgorithm}. Using a modification of Dhar's burning algorithm, it is possible to obtain an ``activity preserving'' bijection between $q$-reduced divisors and spanning trees of $G$; this was originally discovered (using
different terminology) by Cori and Le Borgne \cite{CoriLeBorgne03}.  
In \S\ref{CoriLeBorgneSection} we formulate the Cori-Le Borgne algorithm in the language of reduced divisors.

\medskip

We then turn to the problem of {\em computing} the $q$-reduced divisor equivalent to a given divisor. Dhar's algorithm shows that one can efficiently {\em check} whether a given divisor is $q$-reduced;
we show in \S\ref{MainAlgorithmSubsection} (specifically Algorithm~\ref{MainAlgorithm}) that one can efficiently {\em find} the 
$q$-reduced divisor equivalent to a given divisor as well.  
Algorithm~\ref{MainAlgorithm} can be viewed as the {\em search} version of Dhar's {\em decision} algorithm.  
The main challenge here is the running-time analysis; we use potential theory (specifically, the function $b_q$) 
to give a bound on the running time of Algorithm~\ref{MainAlgorithm}.
As we have already mentioned, the key point is that after firing a set $A \subseteq V(G) \backslash \{ q \}$,
the value of $b_q$ goes down by exactly the size of $A$; this makes the function $b_q$ a powerful tool for running-time analysis in chip-firing processes. The seemingly different techniques of Tardos \cite{Tardos}, Bj{\"o}rner--Lov{\'a}sz--Shor \cite{Lovasz91}, Chung--Ellis \cite{chung99}, van den Heuvel \cite{Heuvel01}, and Holroyd--Levine--M{\'e}sz{\'a}ros--Peres--Propp--Wilson \cite{Rotor} all give bounds which are specializations of the running-time bound which we derive using $b_q$; see Remark~\ref{TimeAnalRmk}.

\medskip

We next turn to applications of Algorithm~\ref{MainAlgorithm}. The first application (see \S\ref{KirchhoffSection}) is an ``efficient bijective'' proof of Kirchhoff's celebrated matrix-tree theorem (stated in a more canonical way than usual in Theorem~\ref{thm:CanonicalKirchhoff}). The efficient bijective matrix-tree theorem provides a new approach to the 
{\em random spanning tree} problem, which we state in \S\ref{SpanningTreesSubsection}. This problem has been extensively studied in the literature and there are essentially two known types of algorithms for it: determinant based algorithms (e.g. \cite{Guenoche83,CMN96,Kulkarni90}) and random walk based algorithms (e.g. \cite{Broder89,Aldous90, Wilson96, KM09} and \cite[Chapter 4]{LyonsPeres09}).  See Remark~\ref{RandomTreeSellingPoints} for possible advantages of our new approach.

\medskip

Finally, we study some analogous questions in the context of {\em metric graphs} (or ``abstract tropical curves'').
While potential theory on finite graphs can be developed purely in the context of linear algebra, potential theory on metric graphs is more conveniently formulated in terms of measure theory. We assume in Appendix~\ref{MetricGraphSection} that the reader is familiar with some basic facts concerning potential theory on metric graphs as explained in \cite{BX06, BakerRumely}.  Our main new result is Theorem~\ref{EnergyThm} (the metric graph analogue of Theorem~\ref{energyThm}); using it, we give a new proof of the existence and uniqueness of reduced divisors in the metric graph setting. We conclude with a discussion of Dhar's algorithm for metric graphs and Luo's search version of this theorem \cite{Luoye}.

\medskip

The paper is structured as follows. 
In \S\ref{NotationSection} we fix our notation.
In \S\ref{PotentialTheorySection} we recall some basic facts from potential theory on graphs and define the energy pairing. 
The functionals $\E_q$ and $b_q$ are introduced in \S\ref{ChipPotSec} and the interplay between chip-firing dynamics and potential theory on graphs is studied.
The algorithmic applications of this interplay, most notably Algorithm~\ref{MainAlgorithm}, are discussed in \S\ref{ComputingReducedSection}. 
Some applications of Algorithm~\ref{MainAlgorithm}, including an ``efficient bijective'' proof of Kirchhoff's matrix-tree theorem and a new algorithm for the random spanning tree problem, are discussed in \S\ref{applicationsSec}. Finally, in Appendix~\ref{MetricGraphSection} we extend some our results to metric graphs.


\section{Notation and Terminology}
\label{NotationSection}

Throughout this paper, a {\em graph} will mean a finite, connected, unweighted multigraph with no loop edges. The set of vertices of a graph $G$ is denoted by $V(G)$ and the set of edges by $E(G)$. We let $n = |V(G)|$ and $m = |E(G)|$. For $A \subseteq V(G)$ and $v \in A$, we denote by $\outdeg_A(v)$ the number of edges between $v$ and $V(G)\backslash A$.

\medskip

Let $\Div(G)$ be the free abelian group generated by $V(G)$. An element 
$\sum_{v \in V(G)} a_v (v) \in \Div(G)$ is called a {\em divisor} on $G$.  The coefficient $a_v$ of $(v)$ in $D$ is denoted by $D(v)$. For $D \in \Div(G)$, let $\deg(D) = \sum_{v \in V(G)} D(v)$ and let $\Div^0(G)$ be the subgroup of $\Div(G)$ 
consisting of divisors of degree zero. We denote by $\M(G) = \Hom(V(G), \ZZ) $ the group of integer-valued functions on the vertices. For $A \subseteq V(G)$,  $\chi_A \in \M(G)$ denotes the $\{0,1\}$-valued characteristic function of $A$; note that $\{\chi_{\{v\}}\}_{v \in V(G)}$ generates $\M(G)$.

\medskip

The {\em Laplacian operator} $\Delta : \M(G) \to \Div(G)$ is defined by 
$\Delta(f) = \sum_{v \in V(G)} \Delta_v(f) (v)$, where 
\[
\Delta_v(f) = \sum_{\{v,w\} \in E(G)} (f(v) - f(w)).
\]

\medskip

This definition naturally extends to all rational or real-valued functions on vertices.

\medskip

Let $\Prin(G)$ (the group of {\em principal divisors}) 
be the image of the  Laplacian operator $\Delta : \M(G) \to \Div(G)$.
It is easy to see that $\Prin(G) \subseteq \Div^0(G)$ and that both $\Prin(G)$ and $\Div^0(G)$ are free abelian groups of rank $n-1$.  As a consequence, the quotient group
\[
\Jac(G) = \Div^0(G) / \Prin(G)
\]
is finite. Following \cite{BacherHN97}, $\Jac(G)$ is called the {\em Jacobian} of $G$. 

\medskip 

Let $\{ v_1,\ldots, v_n \}$ be a labeling of $V(G)$. With respect to this labeling, the {\em Laplacian matrix} $Q$ associated to $G$ is the $n \times n$
matrix $Q =(q_{ij})$, where $q_{ii}$ is the degree of vertex $v_i$ and $-q_{ij}$ ($i \neq j$) is the number of edges connecting $v_i$ and $v_j$.
It is well-known (and easy to verify) that $Q$ is symmetric, has rank $n-1$, and that the kernel of $Q$ is spanned by $\mathbf{1}$, the all-ones vector (see, e.g., \cite{BiggsBook93,Bollobas98}).

\medskip

The labeling  $\{ v_1,\ldots, v_n \}$ of $V(G)$ induces isomorphisms between the abelian groups $\Div(G), \M(G)$, and the group of $n \times 1$ column vectors with integer coordinates. We use $[D]$ to denote the column vector corresponding to $D \in \Div(G)$ and $[f]$ to denote the column vector corresponding to $f \in \M(G)$. Under these isomorphisms, the Laplacian operator $\Delta : \M(G) \to \Div(G)$ corresponds to the matrix $Q$ thought of as a homomorphism $Q: \ZZ^n \to \ZZ^n$, i.e., for $f \in \M(G)$ we have $[\Delta(f)] = Q [f]$.

\section{Potential theory}
\label{PotentialTheorySection}

\subsection{Generalized inverses}
 \label{GeneralizedInverses}
A matrix has an inverse only if it is square and has full rank.
But one can define a ``partial inverse" for any matrix.

\medskip

\begin{definition}
Let $A$ be a matrix. A matrix $L$ satisfying $ALA=A$ is called a \textit{generalized inverse} of $A$.
\end{definition}

\medskip
Every matrix $A$ has at least one generalized inverse. In fact more is true: every matrix has a unique {\em Moore-Penrose pseudoinverse}%
\footnote{The Moore-Penrose pseudoinverse of $A$ is a generalized inverse of $A$ having the following additional properties: (i) $LAL=L$ and (ii) $AL$ and $LA$ are both symmetric.  See \cite{BenIsrael03} for additional details.}.

\medskip
Let $Q$ be the Laplacian matrix of a (connected) graph $G$. Since $Q$ has rank $n-1$, it does not have an inverse in
the usual sense.  But there are several natural ways to obtain generalized inverses for $Q$.  Here are some examples.

\begin{construction} \label{GeneralizeInvConst1}
Fix an integer $1 \leq i \leq n$ and let $Q_i$ be the invertible $(n-1) \times (n-1)$ matrix obtained from $Q$ by deleting $i^{\rm th}$ row and $i^{\rm th}$ column from $Q$ ($Q_i$ is sometimes called the {\em reduced Laplacian} of $G$ with respect to $i$).  Let $L_{(i)}$ be the $n \times n$ matrix obtained from $Q_i^{-1}$ by inserting a row of all zeros after the $(i-1)^{\rm st}$ row and inserting a column of all zeros after the $(i-1)^{\rm st}$ column. Then $L_{(i)}$ is a generalized inverse of $Q$. 
Indeed, one checks that $QL_{(i)}=I+R_{(i)}$, where $I$ is the $n \times n$ identity matrix and $R_{(i)}$ has all $-1$ entries in the $i^{\rm th}$ row and is zero elsewhere; as $R_{(i)}Q=0$, we obtain $QL_{(i)}Q=Q$.
\end{construction}

\begin{construction}\label{GeneralizeInvConst2}
Let $J$ be the $n \times n$ all 1's matrix. Then $Q+\frac{1}{n}J$ is nonsingular and $Q^+= (Q+\frac{1}{n}J)^{-1}-\frac{1}{n}J$ is a generalized inverse of $Q$. In fact, $Q^+$ is the Moore-Penrose pseudoinverse of $Q$, since one 
easily verifies that $QQ^{+}=Q^{+}Q=I-\frac{1}{n}J$ and $Q^{+}QQ^{+}=Q^{+}$.
\end{construction}

One can use the matrices $L_{(i)}$ from Construction~\ref{GeneralizeInvConst1} to obtain other generalized inverses for $Q$:

\begin{construction}

Let $\mu=(\mu_1, \mu_2, \cdots, \mu_n)^T \in \RR^n$  satisfy $\sum_{i=1}^{n}{\mu_i}=1$. Then 
$L_{\mu}=\sum_{i=1}^{n}{\mu_i L_{(i)}}$ is a generalized inverse for $Q$.  The matrix $L_{\mu}$ has the additional property that $L_{\mu}{\mu}=c_{\mu}\mathbf{1}$ for some $c_{\mu} \in \RR$; this follows from the calculation

\[QL_{\mu}{\mu}=(I+\sum_{i=1}^{n}{\mu_i R_{(i)}}){\mu}=\mu-\mu=0 \ .\]

\medskip

If $J$ is the all-1's matrix as in Construction~\ref{GeneralizeInvConst2}, then  $G_{\mu}=L_{\mu}-c_{\mu}J$ is also a generalized inverse and has the additional property that $G_{\mu}\mu=0$. The special case where $\mu_i=1/n$ for all $i$ gives the Moore-Penrose pseudoinverse $Q^+$ from the previous construction.
\end{construction}


\subsection{The $j$-function}
\label{jfunctions}
We can think of a graph $G$ as an electrical network in which each edge is a resistor having unit resistance.

\begin{definition}
For $p,q,v \in V(G)$, let $j_q(p,v)$ denote the electric potential at $v$ if one unit of current enters a network at $p$ and exits at $q$, with $q$ grounded (i.e., zero potential). 
\end{definition}

From a more mathematical point of view, $j_q(p,\cdot)$ is the unique (rational-valued) solution to the Laplace equation $\Delta f = (p)-(q)$
satisfying $f(q)=0$; alternatively, one can define $j_q(p,v)$ to be the $(p,v)$-entry of the matrix  $L_{(q)}$ in Construction~\ref{GeneralizeInvConst1} (see, e.g., \cite{ChinburgRumely,BX06}; note also that $(pv||q)$ in \cite{Biggs97} is the same as our $j_q(p,v)$ up to scaling). 

The following properties of the $j$-function are proved, for example, in \cite{BX06}:

\begin{itemize}
\item $j_q(p,q)=0$.
\item $j_q(p,v)=j_q(v,p)$.
\item $0 \leq j_q(p,v) \leq j_q(p,p)$.
\item $r(p,q)=j_q(p,p)=j_p(q,q)$, where  $r(p,q)$ denotes the effective resistance between $p$ and $q$.
\end{itemize}

\subsection{The energy pairing}
\label{energySection}
\medskip

Let $L$ be any generalized inverse of the Laplacian matrix $Q$. Then the bilinear form 
$\langle  \cdot \, , \cdot \rangle: \; \Div^0(G) \times \Div^0(G) \rightarrow \QQ$ defined by
\begin{equation}\label{EnergyPairingDef}
\langle D_1 , D_2 \rangle = [D_1]^{T} L [D_2]
\end{equation}
is independent of the choice of $L$.  (Indeed, since there are functions $f_i \in \Hom(V(G),\QQ)$
such that $[D_i]=Q[f_i]$ for $i=1,2$, we have $\langle D_1,D_2 \rangle=[f_1]^TQ[f_2]$ and it is easy to check that 
the right-hand side does not depend on the choice of $f_1,f_2$.)
We call the canonical bilinear form $\langle  \cdot \, , \cdot \rangle$ the {\em energy pairing} on $\Div^0(G)$.

\begin{lemma} \label{lem:posdef}
The energy pairing is positive definite.
\end{lemma}

\begin{proof}
Let $B$ be the incidence matrix of the graph.  Then $Q=BB^T$, so if $[D]=Q[f]$ we have 
$\langle D,D \rangle=[f]^TBB^T[f]= \| B^T[f]\|_2^2$.
\end{proof}

\begin{definition}
The {\em energy} of a divisor $D \in \Div^0(G)$ is 
\[\mathcal{E}(D)=\langle D , D \rangle= [D]^{T} L [D] \ . \]
\end{definition}

\noindent Since the energy pairing is positive definite, $\mathcal{E}(D) \geq 0$ with equality iff $D=0$. 

\begin{remark}
The name ``energy pairing'' comes from the fact that if $D \in \Div^0(G)$ represents an external current in the network, 
where $D(v)$ units of current enter the network at $v$ if $D(v)>0$ and
$-D(v)$ units of current exit the network at $v$ if $D(v)<0$,
then $\mathcal{E}(D)$ is precisely the total energy dissipated (per unit time) in the network. 

\end{remark}

We emphasize that the energy pairing is independent of the choice of $L$ only because the divisors are assumed to have degree zero.  One can 
extend the energy pairing to arbitrary divisors by fixing a vertex $q$ and defining the {\em $q$-energy pairing} by
\[ \langle D, E \rangle _{q}= \langle D-\deg(D)(q), E-\deg(E)(q) \rangle \] 
for $D,E \in \Div(G)$.

\subsection{The maximum principle}

\begin{lemma}
\label{GenMaxPrin}
Let $f \in \mathcal{M}(G)$. Let $A_{\max}$ (resp. $A_{\min}$) be the set of vertices where $f$ achieve its maximum (resp. minimum) value. Then
\begin{itemize}
\item[(a)] For $v \in A_{\max}$, $\Delta_v(f) \geq \outdeg_{\A_{\max}}(v)$. 
\item[(b)] For $v \in A_{\min}$, $\Delta_v(f) \leq -\outdeg_{\A_{\min}}(v)$. 
\end{itemize}
\end{lemma}

\begin{proof}
For part (a) let $v \in  \A_{\max}$. For an edge $e = vw$, if $w \in A_{\max}$ then $f(v)=f(w)$, and if $w \not\in A_{\max}$ then $f(v) - f(w) \geq 1$. Since $\Delta_v(f)= \sum_{\{v,w\}  \in E(G)} (f(v) - f(w))$, the result follows.
Part (b) follows from part (a) by replacing $f$ with $-f$. 

\end{proof}

One obtains the following well-known corollary:

\begin{corollary}[Maximum principle]
Suppose $f \in \mathcal{M}(G)$ is nonconstant. Then $f$ achieves its maximum (resp. minimum) value at a vertex $v$ for which $\Delta_v(f) >0$  (resp. $\Delta_v(f)<0$).
\end{corollary}


\section{Chip-firing dynamics and potential theory}
\label{ChipPotSec}

\subsection{Chip-firing dynamics on graphs}
 \label{ChipFiringSubsection}

Following \cite{BN1}, we define an equivalence relation $\sim$ (called {\em linear equivalence}) 
on the group $\Div(G)$ as follows:
\begin{definition}
 For $D_1, D_2 \in \Div(G)$, $D_1 \sim D_2$ if and only if $D_1-D_2$ is in the image of $\Delta : \M(G) \to \Div(G)$.
\end{definition}

This equivalence relation is closely related to notion of {\em chip-firing games} or {\em dollar games}  (see, e.g., \cite{Lovasz91,Biggs97,Biggs99,BN1, Dhar90, BTW88}). 
Given a divisor $D \in \Div(G)$, one can view the integer $D(v)$ as the number of {\em dollars} assigned to the vertex $v$.
If $D(v) <0$ then $v$ is said to be {\em in debt}. A {\em chip-firing move} consists of choosing a vertex and having it 
either borrow one dollar from each of its neighbors or give (``fire") one dollar to each of its neighbors.
For $D_1, D_2 \in \Div(G)$, $D_1 \sim D_2$ if and only if starting from the configuration $D_1$ one can reach the configuration $D_2$ through a sequence of chip-firing moves.

\subsection{Chip-firing moves and the energy pairing}

For $D \in \Div(G)$ we define $\E_q(D)=\langle D, D \rangle_q$. The following two propositions relate the energy pairing and chip-firing moves, and will be used in the next section.
\begin{proposition}\label{energyproposition}
\begin{itemize}
\item[(a)] If $E=D+\Delta(f) \in \Div(G)$ for some $f \in \mathcal{M}(G)$, then \[\mathcal{E}_q(E) =\mathcal{E}_q(D)+\sum_{v\in V(G)} (D+E)(v)\cdot f(v) -2 \deg(E)\cdot f(q) \ .\]
\item[(b)] If $E=D-\Delta(\chi_A) \in \Div(G)$ for some  $A \subseteq V(G)\backslash \{q\}$, then \[\mathcal{E}_q(E) =\mathcal{E}_q(D)-\sum_{v\in A} (D+E)(v) \ .\]
If, moreover, $E$ is effective on $A$ (i.e. $E(v) \geq 0$ for $v \in A$), then 
\begin{equation}\label{energyComputations11}
 \mathcal{E}_q(E) \leq \mathcal{E}_q(D)-\lambda(A) 
\end{equation}
where $\lambda(A)$ denotes the size of the $(A,G\backslash A)$-cut (i.e., the number of edges having one end in $A$ and the other end in $V(G) \backslash A$).
\end{itemize}
\end{proposition}

\begin{proof} 
(a) Let $\deg(E)=d$. Then $\deg(D)=d$ as well. We can find a $\QQ$-valued function $g$ so that $[E+D-2d(q)]=Q[g]$.
Then
\[
\begin{aligned}
\ \mathcal{E}_q(E) &=\langle D+\Delta(f)-d(q), D+\Delta(f) -d(q)\rangle \\
&= \mathcal{E}_q(D)+\langle D+E-2d(q), \Delta(f)\rangle
=\mathcal{E}_q(D)+(Q[g])^T L Q[f] \\
&= \mathcal{E}_q(D)+[g]^TQ L Q[f]
=\mathcal{E}_q(D)+[g]^TQ [f] \\
&= \mathcal{E}_q(D)+(Q[g])^T[f]
=\mathcal{E}_q(D)+[D+E-2d(q)]^T[f] \\
& = \mathcal{E}_q(D)+\sum_{v\in V(G)} (D+E)(v)\cdot f(v) -2d \cdot f(q) \ .
\end{aligned}
\]

(b) For the first statement let $f=-\chi_A$ in part (a). $E(v) \geq 0$ for $v\in A$ means that $D(v) \geq \Delta_v(\chi_A)=\outdeg_A(v)$ for $v \in A$. So $\sum_{v\in A} (D+E)(v) \geq \sum_{v\in A}\outdeg_A(v)=\lambda(A)$.
\end{proof}

\begin{remark}

Proposition~\ref{energyproposition} (a) can be used to give a new solution to the well-known {\em Pentagon Problem}%
\footnote{Problem 3, 27th IMO 1986. We refer the reader to \cite{WinklerBook, WegertReiher, Smirnov, AlonKrasikovPeres, Mozes} for some discussions, solutions, and 
generalizations of this problem.}:
 ``To each vertex of a regular pentagon an integer is assigned in such a way that the sum of all five numbers is positive. If three consecutive vertices are assigned the numbers $x,y,z$ respectively, and $y<0$, then the following operation is allowed : the numbers $x,y,z$ are replaced by $x+y, −y , z+y$, respectively. Such an operation is performed repeatedly as long as at least one of the five numbers is negative. Determine whether this procedure necessarily comes to an end after a finite number of steps''.

To see that this process stops for any $n$-cycle, let $D$ be the starting configuration and assume that $s=\deg(D) \geq 1$.  It follows from 
Proposition~\ref{energyproposition} (a) that the quantity
\[
\mathbb{E} (D) = \sum_{q \in V(G)} {\mathcal{E}_q(D)}=\sum_{q, p, v \in V(G)}{D(p)j_q(p,v)D(v)}
\]
goes down by exactly $-2 s\cdot D(v) >0$ after each basic move with $y=D(v)<0$. Since the energy pairing is positive definite, we always have $\mathbb{E} \geq 0$, 
and thus the procedure will necessarily come to an end after a finite number of steps. Note that this method gives a way to compute the number of steps as well. 

\end{remark}

\medskip 

\begin{definition}
Let $\mathbf{1}$ denote the all-1's divisor. For $D \in \Div(G)$ and $q\in V(G)$, we define $b_q(D)=\langle \mathbf{1}, D \rangle _{q}$.
\end{definition}

\begin{remark}
$b_q(D)$ is the ``total potential'' induced by the external current source corresponding to the divisor $D-\deg(D)(q) \in \Div^0(G)$. 

\end{remark}

\begin{proposition}\label{energyproposition2}
\begin{itemize}
 \item [(a)] If $E=D+\Delta(f) \in \Div(G)$ for some $f \in \mathcal{M}(G)$, then 

\begin{equation}\label{energylemma2eq}
b_q(E)=  b_q(D) + \sum_{v \in V(G)}{(f(v)}-f(q))\ .
\end{equation}

\item[(b)] If $E=D-\Delta(\chi_A) \in \Div(G)$ for some  $A \subseteq V(G)\backslash \{q\}$, then 
\begin{equation}
b_q(E) =b_q(D)-|A| \ ,
\end{equation}
where $|A|$ is the cardinality of the set $A$. Thus $b_q(\cdot)$ is a {\em monovariant}%
\footnote{A quantity which either only goes up or only goes down under some process.}.
\end{itemize}

\end{proposition}
\begin{proof}
(a) We can find a $\QQ$-valued function $g$ so that $[\mathbf{1}-n(q)]=Q[g]$, where $n=|V(G)|$.
Then
\[
\begin{aligned}
\ \langle \mathbf{1}, E \rangle _{q} & =  \langle \mathbf{1}, D \rangle _{q} + \langle \mathbf{1},  \Delta(f) \rangle _{q} \\
&=\langle \mathbf{1}, D \rangle _{q} +(Q[g])^T L Q[f]
=\langle \mathbf{1}, D \rangle _{q} +[g]^TQ L Q[f]\\
&=\langle \mathbf{1}, D \rangle _{q} +[g]^TQ [f]
=\langle \mathbf{1}, D \rangle _{q} +(Q[g])^T[f]\\
&=\langle \mathbf{1}, D \rangle _{q} +[\mathbf{1}-n(q)]^T[f] \\
& =\langle \mathbf{1}, D \rangle _{q} + \sum_{v \in V(G)}{(f(v)}-f(q))\ .
\end{aligned}
\]

(b) follows from part (a) by setting $f=-\chi_A$. 
\end{proof}

The $\QQ$-valued function $g: V(G) \to \QQ$ in the proof of Proposition~\ref{energyproposition2} (a) can be computed 
explicitly, and this gives a useful formula for $b_q$:

\begin{lemma}
\label{g_lemma}
Let $g_q: V(G) \to \QQ$ be the unique function such that $\Delta(g_q) = \sum_v (v) - n(q)$ and $g_q(q)=0$. 
Then:
\begin{enumerate}
\item[(a)]  $g_q(v)=\sum_{p\in V(G)}{j_q(p,v)}$ . 
\item[(b)] For any divisor $D \in \Div(G)$,  
\begin{equation}
\label{eq:b_identity}
b_q(D)=\sum_{v}{g_q(v)D(v)}=\sum_{v}\sum_{p}{j_q(p,v)D(v)} \ .
\end{equation}
In particular, if $D(v) \geq 0$ for $v \ne q$, then $b_q(D) \geq 0$. 
\end{enumerate}
\end{lemma}

\begin{proof}
Part (a) is easy and is left as an exercise.  
Part (b) follows from Construction~\ref{GeneralizeInvConst1} and the definition of the energy pairing. 
Alternatively, let $[\mathbf{1}-n(q)]=Q[g_q]$ and $[D-d(q)]=Q[f]$ where $\deg(D)=d$.  Then
\[
\begin{aligned}
\ \langle \mathbf{1}, D \rangle _{q} &= (Q[g_q])^T L Q[f]\\
&=[g_q]^TQ[f] \\
&=[g_q]^T [D-d(q)]
=\sum_{v}{g_q(v)D(v)} \\
&=\sum_{v}\sum_{p}{j_q(p,v)D(v)}\ . 
\end{aligned}
\]
The second statement follows because $j_q(p,v) \geq 0$ and  $j_q(p,q)=0$.
\end{proof}

\subsection{Reduced divisors}
 \label{ReducedDivisorsSubsection}

A nice set of representatives for equivalence classes of divisors are given by the ``reduced divisors''. 
\begin{definition}
Fix a vertex $q \in V(G)$. A divisor $D \in \Div(G)$ is called {\em $q$-reduced}%
\footnote{Reduced divisors are essentially the same thing as {\em $G$-parking functions} \cite{PostnikovShapiro04} or 
{\em superstable configurations} \cite{Rotor}.}
if it satisfies the following two conditions:
\begin{itemize}
\item[(i)] $D(v) \geq 0$ for all $v \in V(G)\backslash \{q\}$.
\item[(ii)] For every non-empty subset $A \subseteq V(G)\backslash \{q\}$, there exists a vertex $v \in A$ such that $D(v) < \outdeg_A(v)$.
\end{itemize}
\end{definition}

In other words, every vertex outside $q$ is nonnegative but simultaneously firing all the vertices in any non-empty subset 
$A$ of $V(G)$ which is disjoint from $q$ will result in some vertex becoming negative.

\medskip

The significance of reduced divisors comes primarily from the fact that for every $D \in \Div(G)$, there is a {\em unique} $q$-reduced divisor $D'$ such that $D' \sim D$.
This basic fact was discovered independently (in different guises) by several different authors (see, e.g., \cite{Gabrielov93,CoriRossinSalvy02,PostnikovShapiro04,BN1}). 
We give a new proof of this result in Corollary~\ref{energyCor} below.

\medskip

We wish to study reduced divisors from a potential-theoretic point of view.  Fix a distinguished vertex $q$ and define
\[ |D|_q = \{ E \in \Div(G) \, | \, E \sim D, \, E(v) \geq 0 \text{ for all } v\ne q\} \ . \]

\begin{lemma} \label{lem:Dq}
For every $D \in \Div(G)$ and any vertex $q$, the set $|D|_q$ is non-empty. 
\end{lemma}
\begin{proof}
Pick an ordering $\prec$ on $V(G)$ with the property that $q$ is the first vertex in the ordering and every $v \ne q$ has a neighbor $w$ with $w \prec v$. Starting from the last vertex in the ordering, we can inductively make all vertices other than $q$ effective by replacing $D$ with $D-k\Delta(\chi_w)$ for some neighbor $w \prec v$ and some 
sufficiently large integer $k$. 
\end{proof}

\begin{lemma}[Principle of least action] \label{leastActionPrinc}
Let $D$ be a $q$-reduced divisor. Assume $E \sim D$ and write $D=E+\Delta(f)$.
\begin{itemize}
\item[(a)] If $E \in |D|_q$, then $f(v) \leq f(q)$ for all $v \in V(G)$.
\item[(b)] If $E+\Delta(g) \in |D|_q$, then $f(v)-f(q) \leq g(v)-g(q)$ for all $v \in V(G)$.
\end{itemize}

\end{lemma}
\begin{proof}
Part (a) is a consequence of Lemma~\ref{GenMaxPrin}. If $f$ does not achieve its global maximum at $q$, then $A_{\max} \subseteq V(G)\backslash \{ q \}$ and for all $v\in A_{\max}$ we have $\Delta_v(f) \geq \outdeg_{\A_{\max}}(v)$. 
Since $D$ is $q$-reduced, there must be a vertex $u \in A_{\max}$ such that $D(u) < \outdeg_{\A_{\max}}(u)$.  But then 
$E(u)=D(u)- \Delta_u(f)<0$, contradicting the assumption that $E \in |D|_q$.

For (b), let $E'=E+\Delta(g)$. Then $D=E'+\Delta(f-g)$ and part (a) gives $f(v)-g(v) \leq f(q)-g(q)$.
\end{proof}

\begin{theorem} \label{energyThm}
Fix $q \in V(G)$ and let $D \in \Div(G)$.  
Then $D$ is $q$-reduced if and only if $D \in |D|_q$ and $\mathcal{E}_q(D) < \mathcal{E}_q(D')$ for all
$D' \neq D$ in $|D|_q$.
\end{theorem}

\begin{proof}

If $D$ is $q$-reduced, then $D \in |D|_q$. Let $E \in |D|_q$. Then $\mathcal{E}_q(D) \leq \mathcal{E}_q(E)$; write $D=E+\Delta(f)$ with $f(q)=0$. By Lemma~\ref{leastActionPrinc} (a), we have $f(v) \leq 0$. By Proposition~\ref{energyproposition} (a), we have
\[ \mathcal{E}_q(D) = \mathcal{E}_q(E)+\sum_{v \ne q} (D+E)(v)\cdot f(v)  \leq \mathcal{E}_q(E) \ .\]

Now assume $D \in |D|_q$ and $\mathcal{E}_q(D) \leq \mathcal{E}_q(E)$  for all $E \in |D|_q$ but $D$ is not $q$-reduced. Then there exists a non-empty set  $A \subseteq V(G)\backslash \{q\}$ such that $D_1=D-\Delta(\chi_A) \in |D|_q$, so Proposition~\ref{energyproposition} (b) implies

\[
 \mathcal{E}_q(D_1) = \mathcal{E}_q(D)-\lambda(A) \leq  \mathcal{E}_q(D)-1 \ .
\]

\medskip

It follows that if $\mathcal{E}_q(D_1)=\mathcal{E}_q(D_2) \leq \mathcal{E}_q(E)$ for all $E \in |D|_q$,  then both $D_1$ and $D_2$ are $q$-reduced. By Lemma~\ref{leastActionPrinc} (a), if $D_2=D_1+\Delta(f)$ with $f(q)=0$, then $f(v)\leq 0$ for all $v \ne q$. Similarly  $-f(v)\leq 0$ for all $v \ne q$, so $f=0$ and $D_1=D_2$.
\end{proof}

\begin{corollary}\label{energyCor}
Fix $q \in V(G)$ and let $D \in \Div(G)$.  
Then there is a unique $q$-reduced divisor $D' \in \Div(G)$ linearly equivalent to $D$.  
\end{corollary}

\begin{proof}
It follows from Lemma~\ref{lem:Dq}, Proposition~\ref{energyproposition} (b), and Lemma~\ref{lem:posdef} that we may choose $D' \in \Div(G)$ such that 
$\mathcal{E}_q(D') \leq \mathcal{E}_q(D'')$  for all $D'' \in |D|_q$.
By Theorem~\ref{energyThm}, $D'$ is the unique $q$-reduced divisor linearly equivalent to $D$.
\end{proof}

An analogue of Theorem~\ref{energyThm} holds with $\mathcal{E}_q$ replaced by $b_q$:

\begin{theorem} \label{energyThm2}
Fix $q \in V(G)$ and let $D \in \Div(G)$.  
Then $D$ is $q$-reduced if and only if $D \in |D|_q$ and $b_q(D) < b_q(D')$ for all
$D' \neq D$ in $|D|_q$.
\end{theorem}

\begin{proof}
Note that by Lemma~\ref{g_lemma} (b), the function $b_q$ does have a minimum in $|D|_q$. The rest of the proof mirrors the proof of Theorem~\ref{energyThm}.

If $D$ is $q$-reduced, then $D \in |D|_q$ by definition. Let $E \in |D|_q$ be another divisor. Write $D=E+\Delta(f)$ with $f(q)=0$. By Lemma~\ref{leastActionPrinc} (a) we have $f(v) \leq 0$. Now, by Proposition~\ref{energyproposition2} (a), we have
\[   b_q( D) = b_q(E)+\sum_{v \ne q} f(v)  \leq  b_q(E) \ .\]

Now assume $D \in |D|_q$ and $b_q(D) \leq b_q(E)$  for all $E \in |D|_q$ but $D$ is not $q$-reduced. Then there exists a non-empty set  $A \subseteq V(G)\backslash \{q\}$ such that $D_1=D-\Delta(\chi_A) \in |D|_q$. But then Proposition~\ref{energyproposition2} (b) gives the contradiction

\[
b_q( D_1) = b_q(D) -|A| \leq   b_q(D)-1 \ .
\]

\medskip

It follows that if $b_q( D_1)=b_q( D_2) \leq b_q(E)$ for all $E \in |D|_q$, then both $D_1$ and $D_2$ are $q$-reduced. By Lemma~\ref{leastActionPrinc} (a), if $D_2=D_1+\Delta(f)$ with $f(q)=0$, then $f(v)\leq 0$ for all $v \ne q$. Similarly  $-f(v)\leq 0$ for all $v \ne q$, so $f=0$ and $D_1=D_2$.
\end{proof}

\begin{remark}
Theorem~\ref{energyThm2} remains true if $b_q(D)=\langle \mathbf{1}, D \rangle_q$ is replaced by $\langle \mathbf{h}, D \rangle_q$ for any ``$\RR$-divisor'' $\mathbf{h}$ (i.e. $h(v) \in \RR$), provided that  $\mathbf{h}(v) > 0$ for $v \ne q$.

\end{remark}


\section{Algorithmic aspects of reduced divisors}
\label{ComputingReducedSection}

\subsection{Dhar's algorithm}
 \label{DharBurningAlgorithm}

Let $D$ be a divisor on the graph $G$.  In order to check whether or not $D$ is $q$-reduced using the definition,
one needs to check for {\em all subsets} $A \subseteq V(G)\backslash \{q\}$
whether or not there is a vertex $v \in A$ such that $D(v) < \outdeg_A(v)$.   
But there is in fact a much more efficient procedure called {\em Dhar's burning algorithm} (after Dhar \cite{Dhar90}).

\medskip

The idea behind Dhar's algorithm is as follows.
Think of the edges of $G$ as being made of a flammable material.
A fire starts at vertex $q$ and proceeds along each edge adjacent to $q$.
At each vertex $v \neq q$, there are $D(v)$ firefighters, each of whom can control fires in a single direction (i.e., edge)
leading into $v$.
Whenever there are fires approaching $v$ in more than $D(v)$ directions, the fire burns through $v$ and 
proceeds to burn along all the other edges incident to $v$.
The divisor $D$ is $q$-reduced iff the fire eventually burns through every vertex of $G$.

\medskip

More formally, Dhar's algorithm is stated in Algorithm~\ref{DharAlg}.

\begin{algorithm}
\caption{Dhar's Burning Algorithm}
\KwIn{A divisor $D\in \Div(G)$, and a vertex $q \in V(G)$.}
\KwOut{TRUE if $D$ is $q$-reduced, and FALSE if $D$ is not $q$-reduced.}
\BlankLine
\lIf{$D(v) < 0$ {\rm for some} $v \in V(G)\backslash \{q\}$}{output FALSE and Stop.}

Let $A_0=V(G)$ and $v_0=q$.

\For{$1 \leq i \leq n-1$}{Let $A_i=A_{i-1}\backslash\{v_{i-1}\}$. 

\lIf{{\rm for all } $v \in A_i$, $D(v) \geq \outdeg_{A_i}(v)$}{output FALSE and Stop.}

\lElse{let $v_i \in A_i$ be any vertex with $D(v_i) < \outdeg_{A_i}(v_i)$.}
}
Output TRUE.
\label{DharAlg}
\end{algorithm}

\medskip

The complexity of Dhar's algorithm is $O(n^2)$: there are at most $n$ iterations, and at most $n$ inequalities are tested in each iteration.

\subsection{The Cori-Le Borgne algorithm}
\label{CoriLeBorgneSection}

Using a modification of Dhar's burning algorithm, it is possible to obtain an ``activity preserving'' bijection between $q$-reduced divisors (of a given degree $d$) on $G$ and spanning trees of $G$.  This is more or less just a restatement of the work of 
Cori and Le Borgne in \cite{CoriLeBorgne03} in the language of reduced divisors; however, by using the Cori-Le Borgne 
algorithm in conjunction with the results of \S\ref{PotentialTheorySection} and \S\ref{ChipPotSec}, we are able to obtain new results.

\medskip

The idea behind the Cori-Le Borgne algorithm is as follows.
In our original formulation of Dhar's algorithm, we burned through multiple edges at once.
We now use an ordering of $E(G)$ to break ties and implement a ``controlled burn'' in which only one
edge at a time is burnt.

\medskip

Thus, fix a total order on $E(G)$ and suppose we are given a $q$-reduced divisor $D$.
We run Dhar's burning algorithm on $D$, starting with a fire at $q$.
However, any time there are multiple unburnt edges which are eligible to burn, we always choose the {\em smallest}
one.  Whenever the fire burns through a vertex $v$, we {\em mark} the edge along which the fire traveled just before burning through $v$.  Since $D$ is $q$-reduced, the fire eventually burns through every vertex of $G$.
The set of marked edges is connected, has cardinality $n-1$, and covers all vertices and thus forms a {\em spanning tree} $T_D$ of $G$.
We thus obtain an association $\{ {\rm reduced \; divisors} \} \leadsto \{ {\rm spanning \; trees} \}$
(Algorithm~\ref{Reduced2Trees}).

\medskip

The remarkable fact discovered by Cori and Le Borgne is that this association is a {\em bijection}.
The inverse map is also completely explicit and can be described as follows (using the same total order on $E(G)$).
Suppose we are given a spanning tree $T$ in $G$.
A controlled burn starts at the vertex $q$ and, as before,
any time there are multiple unburnt edges eligible to burn we choose the smallest one.
The difference is that now the firefighters at $v$ can control incoming fires in every direction except for those
corresponding to edges of $T$.
Thus the fire burns through a vertex $v \neq q$ exactly when it travels along an edge $e \in T$
from some (burnt) vertex $w$ to $v$. At the moment when $v$ is burned through, we set $D(v)$ equal to 
$|\{ {\rm burnt \; edges \; adjacent \; to \; } v \}| -1$. 
Eventually the fire will burn through every vertex and a nonnegative integer $D(v)$ will have been 
assigned to each vertex $v \neq q$. The value of $D(q)$ is determined by requiring that $\deg(D)=d$.
It turns out that the resulting divisor $D$ is $q$-reduced, so we obtain an association
$\{ {\rm spanning \; trees} \} \leadsto \{ {\rm reduced \; divisors} \}$
(Algorithm~\ref{Trees2Reduced}) which one checks is {\em inverse} to  Algorithm~\ref{Reduced2Trees}.

\medskip

\begin{algorithm}
\caption{Reduced divisor to spanning tree.}
\KwIn{\\$G=(V,E)$ is graph with a fixed ordering on $E$, \\$q \in V(G)$, \\$D=\sum_{v}{a_v(v)}$, a $q$-reduced divisor of degree $d$.}
\KwOut{\\$T_D$ a spanning tree of $G$.}

\BlankLine
{\bf Initialization:}
\\$X=\{q\}$ (``burnt'' vertices), 
\\$R=\emptyset$ (``burnt'' edges), 
\\$T=\emptyset$ (``marked'' edges).

\BlankLine

\While{$X \ne V(G)$}{
$f=\min\{ e=\{s,t\} \in E(G) \, | \, e \not\in R,  \, s \in X, t \not\in X \}$,\\
let $v \in V(G)\backslash X$ be the vertex incident to $f$, \\
\If{$a_{v} = |\{e \text{ {\em incident to} } v \,|\, e \in R\}|$}{
$X \leftarrow X \cup \{v\}$,
\\ $T \leftarrow T\cup \{f\}$,
}
$R \leftarrow R\cup \{f\}$
}
ExtAct$=E \backslash R$,\\
ExtPass$=R \backslash T$,\\
Output $T_D=T$.
\label{Reduced2Trees}
\end{algorithm}

\begin{algorithm}
\caption{Spanning tree to reduced divisor}
\KwIn{\\$G=(V,E)$ is graph with a fixed ordering on $E$, \\$q \in V(G)$, \\$T$ a spanning tree of $G$.}
\KwOut{\\$D_T=\sum_{v}{a_v(v)}$, a $q$-reduced divisor of degree $d$.}

\BlankLine

{\bf Initialization:}
\\$X=\{q\}$ (``burnt'' vertices), 
\\$R=\emptyset$ (``burnt'' edges).

\BlankLine

\While{$X \ne V(G)$}{
$f=\min\{ e=\{s,t\} \in E(G) \, | \, e \not\in R,  \, s \in X, t \not\in X \}$,\\
\If{$f \in T$}{
let $v \in V(G)\backslash X$ be the vertex incident to $f$, \\
$a_{v} := |\{e \text{ incident to } v \,|\, e \in R\}|$, \\
$X \leftarrow X \cup \{v\}$\\
}
$R \leftarrow R\cup \{f\}$
}
$a_q:=d-\sum_{v \ne q}{a_v}$,\\
ExtAct$=E \backslash R$,\\
ExtPass$=R \backslash T$,\\
Output $D_T=\sum_{v}{a_v(v)}$.
\label{Trees2Reduced}
\end{algorithm}

\begin{theorem} \label{thm:CL}
The association given by Algorithms~\ref{Reduced2Trees} and \ref{Trees2Reduced} is a bijection. More precisely:
\begin{itemize}
\item[(i)] For any $q$-reduced divisor $D$ of degree $d$, Algorithm~\ref{Reduced2Trees} outputs a spanning tree $T_D$ of $G$. 
\item[(ii)] For any spanning tree $T$, Algorithm~\ref{Trees2Reduced} outputs a $q$-reduced divisor $D_T$ of degree $d$
on $G$. 
\item[(iii)] Algorithms~\ref{Reduced2Trees} and \ref{Trees2Reduced} are inverse to one another: $T_{D_T}=T$ and $D_{T_D} = D$.
\end{itemize}
 Moreover, under the bijection furnished by Algorithms~\ref{Reduced2Trees} and \ref{Trees2Reduced}:
\begin{itemize}
\item[(iv)] The set $R$ is the same at the end of both algorithms.

\item[(v)] The externally active edges%
\footnote{An edge $e \in E \backslash T$ is called {\em externally active} for $T$ if it is the largest element in the unique cycle contained in $T \cup \{e\}$, and is called {\em externally passive} for $T$ if it is not externally active. 
The {\em external activity} of $T$ is the number of externally active edges for $T$ and is denoted by $ex(T)$.} 
 for $T$ are precisely the elements of ExtAct$=E \backslash R$, and the externally passive edges for $T$ are precisely the elements of ExtPass$=R \backslash T$.
\item[(vi)] The degree of $\sum_{v \neq q} a_v (v)$ is equal to $g-ex(T)$, where $g = m-n+1$. Equivalently, $a_q=d-g+ex(T)$.

\end{itemize}
\end{theorem}

\begin{remark}
\begin{itemize}
\item[(1)] The complexity of both Algorithms~\ref{Reduced2Trees} and \ref{Trees2Reduced} is the same as that of Dhar's algorithm (Algorithm~\ref{DharAlg}), namely $O(n^2)$.
\item[(2)] A natural choice for the fixed degree is $d=g$, in which case it follows from Theorem~\ref{thm:CL} (vi) that $a_q=ex(T)$.
\end{itemize}
\end{remark}

\begin{remark} \label{CriticalRemark} 
The problem of giving an explicit bijection between reduced divisors (in the guise of $G$-parking functions) and spanning trees has been studied in several previous works (see, e.g., \cite{CoriLeBorgne03,ChebikinPylyavskyy05,BensonTetali08}). 
There are also a number of bijections in the literature between {\em $q$-critical configurations} and spanning trees
(see \cite{Gabrielov93,WinklerBiggs97,BiggsTutte99,CoriRossinSalvy02,PostnikovShapiro04, MajumdarDhar}). For a fixed vertex $q$,  $q$-critical configurations provide another set of representatives for equivalence classes of divisors (see, e.g., \cite{Biggs97, Biggs99}). There is a simple relationship between reduced and critical divisors: $D$ is $q$-reduced if and only if $K^{+}-D$ is $q$-critical, where $K^{+}=\sum_{v \in V(G)}{(\deg(v)-1)(v)}$ \cite{BN1}. 
\end{remark}

\subsection{Computing the reduced divisor}
 \label{MainAlgorithmSubsection}
Recall that computing the $q$-reduced divisor equivalent to some divisor $D$ can be viewed as 
the solution to a linear (Theorem~\ref{energyThm2}) or quadratic (Theorem~\ref{energyThm}) integer programming problem.  
Dhar's algorithm (which runs in time $O(n^2)$) 
shows that one can efficiently {\em check} whether a given divisor is the solution to the corresponding integer programming problem.  Next we show that in fact one can {\em find} the solution efficiently as well.

\medskip

Fix a base vertex $q \in V(G)$.  Given a divisor $D \in \Div(G)$, Algorithm~\ref{MainAlgorithm} below 
efficiently%
\footnote{``Efficient'' in this context means that the running time will be polynomial in $m$ and $n$ with only 
$\log(\deg(D))$-bit computations involved.}  
computes the $q$-reduced divisor $D' \sim D$. The idea behind the algorithm is as follows.
Starting with a divisor $D$, the first step is to replace $D$ with an equivalent divisor whose coefficients are ``small''.  This is accomplished by the simple trick of replacing 
$[D]$ with $[D]-Q \lfloor L_{(q)}[D] \rfloor$, where $L_q$ is as in Construction~\ref{GeneralizeInvConst1} and
$\lfloor \cdot \rfloor$ denotes the coordinate-wise floor function. 
The second step is to make the divisor effective outside $q$. This is done by having negative vertices borrow from their neighbors in a controlled way.
The third step is to iterate Dhar's algorithm until we reach a $q$-reduced divisor.  More specifically,
if $D$ is not yet reduced then by running Dhar's algorithm on $D$ we obtain a subset $A$ of $V(G)\backslash \{q\}$
such that firing all vertices in $A$ once yields a divisor $D-\Delta(\chi_A)$ which is still effective outside $q$.
Replacing $D$ by $D-\Delta(\chi_A)$ and iterating this procedure, one obtains (after finitely many iterations)
a $q$-reduced divisor.  Moreover, the number of iterations can be explicitly bounded in terms of
the $j$-function (\S\ref{jfunctions}) using formula \eqref{eq:b_identity}.

\medskip

A formal statement of the resulting algorithm appears below (Algorithm~\ref{MainAlgorithm}).

\begin{algorithm}
\caption{Finding the Reduced Divisor}\label{MainAlgorithm}

\KwIn{\\$Q$ is the Laplacian matrix of the graph $G$, \\$D\in \Div(G)$, \\$q \in V(G)$.}
\KwOut{\\$D' \sim D$ the unique $q$-reduced divisor equivalent to $D$. }
\BlankLine
\textit{(Step 1)} 

Find the generalized inverse $L_{(q)}$ of $Q$, as in Construction~\ref{GeneralizeInvConst1}. Compute the divisor $[D']=[D]-Q \lfloor L_{(q)}[D] \rfloor$.
\BlankLine
 \textit{(Step 2)} 

\lWhile{{\em there exists $v \neq q$ with }$D'(v)<0$}{$[D'] \leftarrow [D']+Q[\chi_{\{v\}}]$.}
\BlankLine
\textit{(Step 3)} 

Let $A_0=V(G)$, $v_0=q$, and $i=1$.

\While{$i \leq n-1$}{
Let $A_i=A_{i-1}\backslash\{v_{i-1}\}$.

\lIf{{\em there exists $v_i \in A_i$ such that} $D'(v_i) < \outdeg_{A_i}(v_i)$}{$i \leftarrow i+1$.}

\lElse{$[D'] \leftarrow [D']-Q[\chi_{A_i}]$. Reset $i=1$.}
}

\end{algorithm}

\medskip
\textbf{Correctness of Algorithm~\ref{MainAlgorithm}:}
\medskip

Assume for the moment that the algorithm actually terminates and produces an output. It is easy to see that the output is 
linearly equivalent to $D$. Also, the output passes Dhar's algorithm and therefore is $q$-reduced (in fact, as discussed above, one can view Step 3 as an iterated Dhar's algorithm). Therefore, for the correctness of the algorithm, we only need to show that it terminates (which follows {\it a posteriori} from the efficiency analysis below).

\medskip

\textbf{Efficiency of Algorithm~\ref{MainAlgorithm}:}
\begin{proposition}\label{DegreeProp}
If $[D']=[D]-Q \lfloor L_{(q)}[D] \rfloor$, then $|D'(v)|<\deg(v)$ for all $v \neq q$. 
\end{proposition}
\begin{proof}
Recall from Construction~\ref{GeneralizeInvConst1} that $QL_{(q)}=I+R_{(q)}$, where $I$ is the identity matrix and $R_{(q)}$ has $-1$ entries in $q^{\rm th}$ row and is zero elsewhere. Therefore $[D]=QL_{(q)}[D]+\deg(D) \cdot \mathbf{e}_q$, where $\mathbf{e}_q$ is the column vector which is $1$ in position $q$ and zero elsewhere.  Now
\[
\begin{aligned}
\ [D'] &=[D]-Q \lfloor L_{(q)}[D]\rfloor\\
&=Q(L_{(q)}[D]-\lfloor L_{(q)}[D]\rfloor)+\deg(D) \cdot \mathbf{e}_q\\
&=Q\mathbf{f}+\deg(D) \cdot \mathbf{e}_q,\\
\end{aligned}
\]
where $\mathbf{f}=L_{(q)}[D]-\lfloor L_{(q)}[D]\rfloor$ is a vector with entries in $[0,1)$. It is now easy to show that the absolute values of the entries of $Q\mathbf{f}$ are bounded by the degree of the corresponding vertices.
\end{proof}

\begin{remark}
Computing the generalized inverse $L_{(q)}$ takes time at most $O(n^{\omega})$, where $\omega$ is the exponent for matrix multiplication (currently $\omega=2.376$ \cite{CoppersmithWinograd}). Notice that this computation is done only once. The second computation in Step (1) can be done using $O(n^2)$ operations (multiplication and addition). For bit complexity, one can check that the denominators appearing in the generalized inverse $L_{(q)}$ are annihilated by the exponent of the Jacobian group. The exponent is bounded above by the number of spanning trees of the graph. If we allow at most $c$ parallel edges then there are at most $c^{n-1}\cdot n^{n-2}$ spanning trees. Moreover, one can also show that the absolute value of the entries of $L_{(q)}$ are bounded above by $R_{\max}$, the maximum effective resistance between any two vertices of the graph. Therefore all integers in the algorithm can be represented in $O(n \cdot \log{cn})$ bits.
\end{remark}

Now we will use our potential theoretic techniques to bound the number of chip-firing moves in Algorithm~\ref{MainAlgorithm}. As we will see, several different bounds in the literature can be obtained as special cases or corollaries of our general potential theory bound.  

\begin{proposition}\label{time}
\begin{itemize}
\item [(a)] Let $D_1$ be the output of Step 1 of Algorithm~\ref{MainAlgorithm}. Then Step 2 of Algorithm~\ref{MainAlgorithm} terminates in at most $b_q(K^{+}-D_1)$ borrowing moves, where $K^{+}=\sum_{v \in V(G)}{(\deg(v)-1)(v)}$. 
\item[(b)] Let $D_2$ be the output of Step 2 of Algorithm~\ref{MainAlgorithm}. Then Step 3 of Algorithm~\ref{MainAlgorithm} 
terminates in at most $b_q(D_2)$ firing moves.
\item[(c)] Algorithm~\ref{MainAlgorithm} terminates in fewer than 
\begin{equation}
\label{eq:MainAlgBound}
3\sum_{v}\sum_{p}{j_q(p,v)\deg(v)}
\end{equation}
chip-firing moves.
\end{itemize}
\end{proposition}

\begin{proof}
(a) By Proposition~\ref{DegreeProp}, $|D_1(v)|<\deg(v)$ for all $v \neq q$. So for any  $v \neq q$ with $D_1(v)<0$, only one borrowing is needed to make the vertex positive. Moreover, the resulting positive number will be less than $\deg(v)$. This fact, together with Proposition~\ref{DegreeProp}, guarantees that the output of Step 2 satisfies $0 \leq D_1(v)<\deg(v)$ for all $v \neq q$. The result now follows from Lemma~\ref{g_lemma} and Proposition~\ref{energyproposition2} (b); the value of $b_q(\cdot)$ is at least $\sum_{v}{\sum_{p}{j_q(p,v)D_1(v)}}$ on the input of Step 2, and is at most $\sum_{v}{\sum_{p}{j_q(p,v)(\deg(v)-1)}}$. Moreover, with each borrowing $b_q(\cdot)$ increases by 1.

(b) This again follows from Lemma~\ref{g_lemma} and Proposition~\ref{energyproposition2} (b). Note that $D_2(v) <\deg(v)$, and that no vertex $v \ne q$ can become negative in Step 3.

(c) This follows from parts (a) and (b) and the inequalities 
\[b_q(K^{+}-D_1)<2\sum_{v}{\sum_{p}{j_q(p,v)\deg(v)}} \ , \] 
\[b_q(D_2)<\sum_{v}{\sum_{p}{j_q(p,v)\deg(v)}} \ . \]

\end{proof}

\subsection{Comparison to other techniques in the literature}
\label{RunningTimeSec}
By basic properties of the $j$-function (see \S~\ref{jfunctions}), we have
\[
3\sum_{v}\sum_{p}{j_q(p,v)\deg(v)} \leq 3(n-1)\sum_{v}{r(v,q)\deg(v)} \ .
\]
By Proposition~\ref{time} (specifically \eqref{eq:MainAlgBound}),
it follows that Algorithm~\ref{MainAlgorithm} terminates in fewer than
\begin{equation} \label{MainTime}
3(n-1)\sum_{v}{r(v,q)\deg(v)}
\end{equation}
chip-firing moves.

The bound \eqref{MainTime} can be computed in matrix multiplication time $O(n^{\omega})$ (currently $\omega=2.376$) because $j_q(p,v)$ is simply the $(p,v)$-entry of the matrix  $L_{(q)}$ in Construction~\ref{GeneralizeInvConst1}. 

\begin{remark} \label{TimeAnalRmk}
There are several ways to bound the expression in \eqref{MainTime} in terms of more familiar invariants of the graph.  
For example:

\begin{enumerate}
\item[(1)] Let $R_{\text{max}}$ be the maximum effective resistance between vertices of $G$ and let $\Delta_{\text{max}}$ be the maximum degree of a vertex in $G$. Then \eqref{MainTime} is bounded above by
\[
3(n-1)R_{\text{max}}\sum_{v \ne q}{\deg(v)}  \ ,
\]
which is, in turn, bounded above by
\[
 3(n-1)^2R_{\text{max}}\Delta_{\text{max}} \ .  
\]
These estimates give a factor $n$ improvement over the bound for the running time of Algorithm~\ref{MainAlgorithm} 
which could be derived using the technique in \cite{Rotor} by Holroyd, Levine, M{\'e}sz{\'a}ros, Peres, Propp, and Wilson.

\item[(2)] Using Foster's network theorem, one can show that 
\[ r(v,q) < 3\sum_{v\in V(G)}{(\deg(v)+1)^{-1}}\] (see, e.g., \cite[proof of Theorem 6]{Coppersmith}). 
So another upper bound for the running time of Algorithm~\ref{MainAlgorithm} is
\[ 9(n-1)\sum_{v\in V(G)}{(\deg(v)+1)^{-1}}\sum_{v \ne q}{\deg(v)} \ . \]
This is a  good bound when the graph is close (on average) to being regular.
If one uses the fact that degree of a vertex is at least the edge-connectivity $\lambda$ of the graph, one gets a bound of 
the form $O(n^2m/\lambda)$ for the running time of Algorithm~\ref{MainAlgorithm}. 
This is (up to constant factors) the bound that one can derive from the techniques of van den Heuvel \cite{Heuvel01}.

\item[(3)] Let $\lambda_1$ (called the {\em algebraic connectivity} of $G$) be the smallest 
non-zero eigenvalue of $Q$.  Then $r(p,q) \leq \frac{2}{\lambda_1}$ for every $p,q \in V(G)$.
Indeed, the proof of the Cauchy-Schwarz inequality shows that for any positive semidefinite matrix $L$ with largest eigenvalue $\eta$, and for all vectors $x$ and $y$, $|x^{T}Ly|\leq \eta \|x\|_2 \|y\|_2$;
if we apply this estimate with $x=y=(p)-(q)$, $L=Q^{+}$ (cf. Construction~\ref{GeneralizeInvConst2}), then
$\eta=1/\lambda_1$, and we obtain
\[r(p,q)=\langle (p)-(q), (p)-(q) \rangle \leq \frac{2}{\lambda_1} \ .\]
Therefore an upper bound for the running time of Algorithm~\ref{MainAlgorithm} is
\[ \frac{6(n-1)}{\lambda_1}\sum_{v\ne q}{\deg(v)}. \]
This is the bound that one can derive from the techniques of 
Bj{\"o}rner--Lov{\'a}sz--Shor \cite{Lovasz91} or Chung--Ellis \cite{chung99}.

\item[(4)] By Rayleigh's monotonicity law, we have $R_{\max} \leq \diam(G)$, 
where $\diam(G)$ denotes the {\em diameter} of $G$. Equality holds if and only if $G$ is a path. In fact $R_{max}$ is much smaller than $\diam(G)$ in a general graph.
Another upper bound for the running time of Algorithm~\ref{MainAlgorithm} is
\[ 3(n-1) \diam(G)\sum_{v\ne q}{\deg(v)} \ . \]
This is the bound that one can derive from the techniques of Tardos \cite{Tardos}.
\end{enumerate}
\end{remark}

\begin{remark}
Items (3) and (4) in the previous remark clarify the relationship between the seemingly different approaches of 
Tardos and  Bj{\"o}rner--Lov{\'a}sz--Shor.
\end{remark}


\section{Some applications of the algorithms}
\label{applicationsSec}

\subsection{Bijective matrix-tree theorem}
\label{KirchhoffSection}

Kirchhoff's celebrated {\em matrix-tree theorem} is usually formulated as follows.  Let $G$ be a (connected) graph.
Following the terminology from Construction~\ref{GeneralizeInvConst1}, fix an integer $1 \leq i \leq n$ and let $Q_i$ be the invertible $(n-1) \times (n-1)$ matrix obtained from the Laplacian matrix $Q$ of $G$ by deleting $i^{\rm th}$ row and $i^{\rm th}$ column from $Q$.  

\begin{theorem}[Kirchhoff's matrix-tree theorem \cite{Kirchhoff1847}] \label{thm:Kirchhoff}
The number of spanning trees in $G$ is equal to $|\det(Q_i)|$.
\end{theorem}

Our aim in this section is to give an ``efficient bijective'' proof of Theorem~\ref{thm:Kirchhoff}.  In order to make sense of this goal,
it is useful to reformulate Kirchhoff's theorem in a more natural way in terms of the Jacobian group%
\footnote{Although we do not need this here, it is worth mentioning that the energy pairing descends to a non-degenerate $\QQ/\ZZ$-valued bilinear form on $\Jac(G)$; it is called the ``monodromy pairing'' in \cite{FarbodDLP09}.} 
$\Jac(G) = \Div^0(G) / \Prin(G)$,
where $\Div^0(G)$ is the subgroup of $\Div(G)$ consisting of divisors of degree zero and  $\Prin(G)$ (the group of principal divisors) is the image of the  Laplacian operator $\Delta : \M(G) \to \Div(G)$.

By elementary group theory (e.g. the theory of the {\em Smith normal form}), one sees that $\Jac(G)$ is the torsion part\footnote{The full cokernel $\Pic(G)= {\Div(G)}/{\Prin(G)}$ is isomorphic to $\ZZ \oplus \Jac(G)$.} of the cokernel of $Q: \ZZ^n \to \ZZ^n$, and the order of $\Jac(G)$ is equal to $|\det(Q_i)|$. 
We may thus reformulate Kirchhoff's theorem as follows:

\begin{theorem}[Kirchhoff's matrix-tree theorem, canonical formulation] \label{thm:CanonicalKirchhoff}
The number of spanning trees in $G$ is equal to $|\Jac(G)|$.
Moreover, there exists an efficiently computable bijection between elements of $\Jac(G)$ and spanning trees of $G$.
\end{theorem}

Note that such a bijection cannot be canonical, as that would imply the existence of a distinguished spanning tree in $G$ corresponding to the identity element of $\Jac(G)$, but it is
clear (think of the case where $G$ is an $n$-cycle) that there is in general no distinguished spanning tree.
Therefore, one needs to make some choices to write down a bijection.

\begin{proof}[Proof of Theorem~\ref{thm:CanonicalKirchhoff}]
By Corollary~\ref{energyCor}, if we fix a vertex $q$ of $G$, there is a unique $q$-reduced divisor representing each class in 
$\Div^0(G)$.  In particular, there is an explicit bijection between $\Div^0(G)$ and the set of $q$-reduced divisors of degree $0$.  If in addition we choose a total order on $E(G)$, then the algorithms in \S\ref{CoriLeBorgneSection} show that there is
a bijection between $q$-reduced divisors of degree $0$ and spanning trees of $G$.

We have shown in \S\ref{MainAlgorithmSubsection} that the unique $q$-reduced representative for each class in 
$\Div^0(G)/\Prin(G)$ can be computed efficiently. 
And in \S\ref{CoriLeBorgneSection} we proved that the bijection between $q$-reduced divisors of degree $0$ and 
spanning trees of $G$ is also efficient. 
\end{proof}

Kirchhoff's matrix-tree theorem is of course a classical result. The main new contribution here is to observe that the bijections between reduced divisors and spanning trees (as described in \S\ref{CoriLeBorgneSection} and Remark~\ref{CriticalRemark}), in conjunction with Corollary~\ref{energyCor}, furnish a simple bijective proof of Kirchhoff's theorem, and moreover this bijection is efficiently computable.



\subsection{Random spanning trees}
\label{SpanningTreesSubsection}

The {\em random spanning tree} problem has been extensively studied in the literature and there are two known types of algorithms: determinant based algorithms (e.g. \cite{Guenoche83,CMN96,Kulkarni90}) and random walk based algorithms (e.g. \cite{Broder89,Aldous90, Wilson96, KM09} and \cite[Chapter 4]{LyonsPeres09}).

Here we give a new deterministic polynomial time algorithm for choosing a random spanning tree in a graph $G$.  

Although the bound we obtain for the running time of our algorithm does not beat the current best known running time 
$O(n^{\omega})$ of \cite{CMN96}, we believe that our algebraic method has some advantages.  For example, it is 
trivial that the output of our algorithm is a uniformly random spanning tree, whereas in other algorithms
(e.g. \cite{CMN96}) this fact is non-trivial and requires proof. See Remark~\ref{RandomTreeSellingPoints} for another
advantage.

\medskip

The idea behind our algorithm is very simple.  Fix a vertex $q \in V(G)$ and a total ordering of $E(G)$.  
The first step in the algorithm is to compute a presentation of $\Jac(G)$ as a direct
sum of cyclic groups; this can be done efficiently by computing the Smith normal form for $Q$.
Once $\Jac(G)$ is presented in this way, it is clear how to select a random element.  Having done so, one 
computes the corresponding $q$-reduced divisor and then the corresponding spanning tree.

This procedure is formalized in Algorithm~\ref{RandomTree} below.

\begin{algorithm}
\caption{Choosing a uniformly random spanning tree}\label{RandomTree}

\KwIn{A graph $G$.}
\KwOut{A uniformly random spanning tree of $G$.}
(1) Fix a vertex $q \in V(G)$ and a total ordering of $E(G)$.
\\(2) Compute the Smith normal form of the Laplacian matrix of G to find:
\\- invariant factors $\{n_1,\ldots,n_s\}$,
\\- generators $\{\mathbf{g}_1, \ldots , \mathbf{g}_s\}$ for $\Jac(G)$ (thought of as elements of $\Div^0(G)$).
\\(3) Choose a random integer $0 \leq a_i \leq n_i-1$ for ($1 \leq i \leq s$).
\\(4) Compute the divisor $D=\sum_{i=1}^{s}{a_i\mathbf{g}_s}$.
\\(5) Use Algorithm~\ref{MainAlgorithm} to find the unique $q$-reduced divisor $D'$ equivalent to $D$.
\\(6) Use Algorithm~\ref{Reduced2Trees} to find the spanning tree corresponding to $D'$.

\end{algorithm}

To our knowledge, the fastest known Smith normal form algorithm (Step (2)) is given in \cite{KaltofenVillard04} and has running
time $(n^{2.697263}\log{\|Q\|})^{1+o(1)}$, where $\|Q\|$, for our application, means the maximal degree of a vertex $\Delta_{\max}$. See also \cite{Saunders} for a fast and practical Smith normal form algorithm.
For the running time of Step (5) see \eqref{eq:MainAlgBound}, \eqref{MainTime}, and Remark~\ref{TimeAnalRmk}. 
Step (6) can be done in $O(n^2)$ steps.

\begin{remark} \label{RandomTreeSellingPoints}
Note that for repeated sampling of random spanning trees in $G$, one has to perform steps (1) and (2) of
Algorithm~\ref{RandomTree} only once. Note also that if there are $N$ spanning trees in $G$, our algorithm uses only $\log_2{N}$ random bits for generating each random spanning tree. Thus our algorithm may have some advantages over existing methods for sampling multiple spanning trees. For example, very few random bits are required in our algorithm to generate pairwise independent spanning trees; to generate $k$ pairwise independent spanning trees, the naive approach would use $k \cdot \log_2{N}$ random bits. But one can use standard methods to pick pairwise independent elements of the group using only $O(\log_2{N})$ random bits. Also, it is possible with our method to sample multiple spanning trees according to joint distributions other that the uniform distribution.  (We thank Richard Lipton for these observations; see \cite{LiptonBlog}).
\end{remark}


\subsection{Other applications}

We list briefly some other applications of our algorithms:

\begin{enumerate}
\item[(1)] (The group law attached to chip-firing games)
If we fix a vertex $q \in V(G)$, then $\Jac(G)$ induces a group structure on the set of $q$-reduced divisors ($G$-parking functions) or $q$-critical divisors of $G$. The latter is called {\em critical group} (or sandpile group) of $G$.  Adding two
elements in one of these groups requires first adding the given divisors as elements of $\Div(G)$, and then finding the unique $q$-reduced or $q$-critical divisor equivalent to the sum. 
Our algorithm for finding $q$-reduced divisors can be used to efficiently compute the group law in these groups.
A different approach for performing the group operation is given in \cite{Heuvel01} using ``oil games''. The problem of finding a ``purely algebraic'' method for computing the sum of two elements of the critical group (and analyzing the running
time of the resulting algorithm) was posed as an open problem by Chung and Ellis in \cite{chung99}.

\medskip

\item[(2)] (Determining whether the dollar game is winnable)
In \cite{BN1}, the authors consider a dollar game played on the vertices of $G$.  
Given a divisor $D$, thought of as a configuration of dollars on $G$, the goal of the game is to get all the vertices out
of debt via borrowing and lending moves, i.e., to find an effective divisor $D'$ linearly equivalent to $D$.
By the proof of Theorem 3.3 in \cite{BN1}, the game is winnable iff the unique $q$-reduced divisor equivalent to
$D$ is effective.
As a corollary, once we can efficiently compute the $q$-reduced divisor associated to a given configuration, we can
efficiently decide whether or not there is a winning strategy, and when there is one we can efficiently compute a sequence of winning moves.

These considerations are related to the Riemann-Roch theorem for graphs from \cite{BN1}.
To any divisor $D\in \Div(G)$ one associates an integer $r(D) \geq -1$, called the {\em rank} of $D$,
such that $r(D) \geq 0$ iff the unique $q$-reduced divisor equivalent to $D$ is effective.
Our algorithm for finding $q$-reduced divisors can therefore be used to efficiently check whether or not $r(D) \geq 0$.
More generally, we can efficiently check whether $r(D) \geq c$ for any fixed constant $c$. It is an open problem to determine whether or not one can compute $r(D)$ itself in polynomial time. For a study of this problem, see \cite{Madhu11}.
\end{enumerate}

\appendix

\section{Metric graphs}
\label{MetricGraphSection}

Our goal in this section is to extend some of the considerations from \S\ref{energySection}, as well as
Theorem~\ref{energyThm} and Corollary~\ref{energyCor}, to the setting of {\em metric graphs}.
The metric graph analogue of Corollary~\ref{energyCor} has already been proved in
\cite{MK08} and \cite{Hladkyetal}, but since this is the main ingredient needed to prove the Riemann-Roch theorem for tropical curves (an important result in tropical geometry), 
it seems worthwhile to present the new proof which follows.
Our main new result is Theorem~\ref{EnergyThm} (the metric graph analogue of Theorem~\ref{energyThm}).

\subsection{Background  on  metric  graphs  and  potential  theory}
	
We assume that the reader is familiar  with  some  basic  facts concerning potential theory on metric graphs; see for example \cite{BX06, BakerRumely}. We recall here the main facts and terminology which we will use.
	
Let $\Gamma$ be a metric graph. Let $\mathcal{C}(\Gamma)$ denote the $\RR$-algebra of continuous real-valued functions on $\Gamma$, let $\CPA(\Gamma)
\subset \mathcal{C}(\Gamma)$ be the vector space consisting of all continuous
piecewise affine functions on $\Gamma$, and let $R(\Gamma)$ be the subgroup 
of $\CPA(\Gamma)$  consisting  of  continuous piecewise  affine  functions  with 
integer slopes (this can be viewed as the space of {\em tropical rational functions} on $\Gamma$, 
cf. \cite{GathmannKerber, MK08}).
	
Let $\Delta$ be the Laplacian operator on $\Gamma$ (see \cite{BX06, BakerRumely}), which takes a certain subspace $\BDV(\Gamma)$ of $\mathcal{C}(\Gamma)$ into the space of measures of total mass zero on $\Gamma$.  (The abbreviation $\BDV$ stands for ``bounded differential variation''.)  More  precisely,  let $\Meas^0(\Gamma)$  be  the  vector  space  of  finite  signed  Borel measures  of total  mass  zero  on  $\Gamma$  and let  $\RR\subset \mathcal{C}(\Gamma)$  denote  the  space  of constant  functions  on  $\Gamma$. Then  the  space  $\BDV(\Gamma)$  is  characterized  by  the property that 
$f \mapsto \Delta(f)$ induces an isomorphism of vector spaces 
	
\[\BDV(\Gamma)/ \RR \longleftrightarrow \Meas^0(\Gamma) \ . \]
	
For a fixed $q  \in \Gamma$, an inverse to $\Delta$ on $\{f \in \BDV(\Gamma)  \, | \, f (q) = 0\}$ is given by 
\begin{equation} 
\nu \mapsto \int_{\Gamma} j_q (x, y)d\nu(y) \in \BDV(\Gamma) 
\end{equation}
where $j_q (x, y)$ is the fundamental potential kernel on $\Gamma$
(defined below, see also  \cite{BakerRumely}).  It  is  shown  in  \cite{BakerRumely} that  $\CPA(\Gamma) \subset \BDV(\Gamma)$  and  that  $\Delta(f)$  is  a discrete measure if 
and only if $f \in \CPA(\Gamma)$. 
	
For $f \in \CPA(\Gamma)$ we have
\[
\Delta (f) =  \sum_{p \in \Gamma}{\sigma_{p} (f)}
\]	
where $-\sigma_p(f)$ is the sum of the slopes of $f$ in all tangent directions emanating from $p$.
(Note that $\Delta(f) = -\divisor(f)$ with the conventions from \cite{GathmannKerber, Hladkyetal}.)
This formula uniquely characterizes $\Delta$ on $\BDV(\Gamma)$, because $\CPA(\Gamma)$ is dense in $\mathcal{C}(\Gamma)$ and for $f \in \CPA(\Gamma)$ and $g \in \BDV(\Gamma)$, we have 
\begin{equation} \label{LaplaceIntegral}
\int_{\Gamma}f \Delta (g) = \int_{\Gamma} g \Delta (f) \ .
\end{equation}
	
(In fact, (\ref{LaplaceIntegral}) holds for all $f, g \in \BDV(\Gamma)$.) 
	
For fixed $q$ and $y$, one can define $j_q (x, y)$ as the unique element
of $\CPA(\Gamma)$ such that $\Delta_x j_q(x, y) = \delta_y(x) - \delta_q(x)$. 
In terms of electrical network theory, if we think of $\Gamma$ as an electrical
network with resistances given by the edge lengths, then $j_q(x, y)$ is the
potential at $x$ when one unit of current enters the network  at $y$  and  exits
at  $q$,  with  reference voltage zero  at $q$. It is  a  basic fact,  proved 
in  \cite{ChinburgRumely} (see  \cite{BX06}  for  an  alternate  proof),  that  $j_q(x, y)$  is  jointly continuous in $x$, $y$, and $q$.
	
There is a canonical bilinear form, called the {\em energy pairing}, on
$\Meas^0(\Gamma)$. It  can be  defined  in many equivalent  ways;  if  $\mu =
\Delta(f)$  and $\nu = \Delta(g)$  are in $\Meas^0(\Gamma)$, then
\begin{equation} \label{energy}
\langle \mu, \nu \rangle = \int f d\nu = \int g d\mu = \int f'(x)g'(x) dx = \int j_q (x, y) d\mu(x)d\nu(y) \ .
\end{equation}
It is proved in \cite [Theorem 10.4] {BakerRumely} that the energy pairing is {\em positive definite}. 
If $\mu \in \Meas^0(\Gamma)$ is a discrete measure, the energy pairing
$\langle \mu, \mu \rangle$ has a nice interpretation in terms of electrical networks.  Write $\mu = \sum a_i
\delta_{p_i} - \sum b_j \delta_{p_j}$ with $a_i , b_j \geq 0$. If $a_i$ units of
current  enter  the  network  at  each $p_i$   and  $b_j$ units of current exit
the network at each $q_j$, then $\langle \mu, \mu \rangle$ is the total energy
(or power) dissipated in the network.  The different formulas for the energy pairing in (\ref{energy})
correspond to the different classical formulas for computing the power (as the sum of $V I$, $I^2R$, etc.).

\subsection{Existence  and  uniqueness  of  reduced  divisors}

Let  $\Div(\Gamma)$ (the group of {\em divisors} on $\Gamma$) be the
free abelian group on $\Gamma$. We  can  identify  a  divisor $D= \sum a_i(p_i)$  on 
$\Gamma$  with  the  discrete  measure $\mu_D :=  \sum a_i \delta_{p_i}$.  
We will frequently  identify  divisors  and  measures  and  will often  not  explicitly 
differentiate  between $D$ and $\mu_D$. 

Let $\Div^0 (\Gamma)$  be  the subgroup
of  divisors  of  degree zero  on  $\Gamma$,  and  let  $\Prin(\Gamma)$  be 
the  subgroup $\{ \Delta (f) \, | \,  f \in  R(\Gamma)\}$  of  $\Div^0 (\Gamma)$ 
consisting  of {\em principal  divisors}. We write $D \sim D'$ if $D - D'$ belongs to $\Prin(\Gamma)$ and say that $D$
and $D'$ are {\em linearly equivalent}. 
	
Fix  $q  \in \Gamma$.  A  divisor  $D  =  \sum a_p(p)$  is  called 
{\em effective} if $a_p \geq 0$  for  all $p$, and is called {\em effective outside $q$}  if
$a_p \geq 0$ for all $p  \ne  q$.  For $D \in \Div(\Gamma)$, we define  the 
complete  linear  system  $|D|$ to  be  the  set  of  all  effective  divisors
$E$  equivalent  to  $D$,  i.e.,  $|D| =  \{E  \in \Div(\Gamma)  \, | \, E  \ge 0,    
E  \sim D\}$. Similarly, we define $|D|_q$ to be the set of divisors equivalent
to $D$ which are effective outside $q$: 
	\[
	|D|_q = \{E \in \Div(\Gamma) \, | \, E(p) \geq 0 , \forall p \ne q,  E \sim D\}
. 
	\]
	
	\begin{definition}
	Fix  $q  \in \Gamma$.     A  divisor $D$  on  $\Gamma$  is  called 
{\em $q$-reduced} if  it  satisfies 
	the following two conditions: 
	\begin{itemize}
	\item[(R1)]  $D$ is effective outside $q$. 
	\item[(R2)]  If $f \in R(\Gamma)$  is  non-constant  and  has  a  global
maximum  at $q$,  then 
\[		D + \Delta(f) \not \in |D|_q \ .  \]
	\end{itemize}
\end{definition}

Note  that  this  is  not  the  usual  definition  of  reduced  divisor 
on  a  metric graph, but it is easily seen to be equivalent to the definition
used in \cite{Amini, Hladkyetal, Luoye}: 
	
\begin{lemma} \label{ReducedDefLem}
	A divisor $D$  is $q$-reduced if and only if it satisfies (R1) and 
	
	(${\rm R2}^\prime$)  For every closed connected set $X  \subseteq \Gamma$  not
containing $q$, there exists a point $p \in \partial X$  such that $D(p) <
\outdeg_X (p)$. 
\end{lemma}
	
\begin{proof}
	Suppose  $D$  is  effective  outside  $q$  and  satisfies  (R2).  Given 
a  closed 
	connected  set  $X$  not  containing  $q$,  construct  a  rational 
function $f \in R(\Gamma)$ 
	which is $0$ on $X$ and $\epsilon$ outside of an $\epsilon$-neighborhood of $X$ for
some sufficiently 
	small  $\epsilon > 0$, with slope $1$ in each 
	outgoing direction from $X$  and $f(q) = \epsilon$ (so $f$  achieves its
maximum at $q$). By  (R2),  there  is  a  point $p \in \Gamma \backslash \{q\}$
such  that  $(D + \Delta(f)) (p)  < 0$. Since $(\Delta(f)) (p) = -\outdeg_X
(p)$, it follows that $p  \in \partial X$  and $D(p) < \outdeg_X (p)$. 
	
	Conversely, suppose $D$  is effective outside $q$  and satisfies (${\rm R2}^\prime$).  
	Given a non-constant function $f  \in R(\Gamma)$ achieving its maximum value at $q$,
let $X$ be a connected component of the set of points where $f$ achieves its minimum. By assumption 
(${\rm R2}^\prime$), there exists $p  \in \partial X$  such that $D(p) < \outdeg_X (p)$. Since $(\Delta (f)) (p) <-\outdeg_X (p)$, we have $(D + \Delta (f)) (p) \leq D(p)-\outdeg_X (p) <	0$ and thus $D + \Delta (f) \not\in |D|_q$.
\end{proof}

From the definition, one sees that if $D$ is a $q$-reduced divisor and
$D+\Delta(f)$ is effective outside $q$ for some non-constant function $f  \in R(\Gamma)$, then $f$ 
cannot have a global maximum at $q$.  It turns out that such an $f$ must in fact have a
global {\em minimum} at $q$. The following lemma is the analogue of Lemma~\ref{leastActionPrinc}.

\begin{lemma} [Principle of least action] \label{LeastActLemma}
 If  $D$  is  $q$-reduced  and  $f  \in R(\Gamma)$  is  a 
rational  function with $D + \Delta(f) \in |D|_q$, then $f$  has a global minimum at $q$. 
\end{lemma}
 
\begin{proof}	
 Suppose  not,  and  let  $X$    be  the  set  of  points  where  $f$
achieves  its minimum  value. Then  $X$ is  a  closed  connected  set  not containing  $q$.   By Lemma \ref{ReducedDefLem},  there  exists  $p \in  \partial X$   such  that  $D(p)  <  \outdeg_X(p)$. On  the 
other hand, we have $\Delta(f)(p) < -\outdeg_X (p)$, and thus $(D +\Delta(f)) (p) < 0$, contradicting the assumption that $D + \Delta(f) \in |D|_q $.
\end{proof}

The  importance  of  reduced  divisors  is  given  by  the  following 
theorem \cite{Hladkyetal,MK08}, which is analogous to 
the corresponding result for finite graphs (Corollary~\ref{energyCor} above) proved in \cite{BN1}:
	
\begin{theorem} \label{UniquenessThm}
 Fix $q \in \Gamma$.  There is a unique $q$-reduced divisor in each linear equivalence class of divisors on $\Gamma$. 
\end{theorem}

We  will  give  a  new  proof  of  this  theorem  based  on  the following  energy minimization  result. For  a  divisor  $D  \in  \Div(\Gamma)$  of degree $d$,  define  the  {\em $q$-energy} of $D$ by 
\[\E_q(D) = \langle D - d(q), D - d(q)\rangle \ . \]

If $D =  \sum a_i (p_i)$ is effective and $p_i \ne q$ for all $i$, then $\E_q(D)$ is
the total energy dissipated in an electrical network in which $a_i$  units of current enter the network at $p_i$ 
and $d =\sum a_i$ units of current exit at $q$. 
	
\begin{theorem} \label{EnergyThm}
Let $D$  be  a  divisor  on $\Gamma$  which  is  effective outside $q$. Then $D$  is $q$-reduced if and only if it has smaller $q$-energy than every other divisor in $|D|_q$. 
\end{theorem}

\begin{proof}
Suppose first that $D$ has minimal $q$-energy among all divisors in $|D|_q$. We  want  to  show  that  $D$  is $q$-reduced.     Let  $f   \in  R(\Gamma)$ be  a  non-constant rational  function with  a  global  maximum  at  $q$.      Without  loss of  generality, we  may  assume  that  $f (q)  =  0$.    We  claim  that  $D' := D + \Delta(f)$  does  not belong to  $|D|_q$  .  Indeed, if $D$   were in  $|D|_q$  , then we would have
\begin{align*}
\E_q(D) &= \E_q(D' ) - 2 \langle D' - d(q), \Delta(f)\rangle + \E_q(\Delta(f))\\ 
&> \E_q(D') - 2f (D')\\
&\geq  \E_q(D') 
\end{align*}
since  $f (D')  \leq 0$  by  our  assumptions  on  $f$    and  $D'$ . This contradicts  the minimality of $\E_q(D)$, and thus $D' \not\in |D|_q$   as claimed. 

For  the  other  direction,   we  need  to  prove  that  if  $D$   is $q$-reduced  and $D'= D + \Delta(f) \in |D|_q$  with  $f  \in R(\Gamma)$ non-constant, then $\E_q(D') > \E_q(D)$. 
As  before  we  may  assume  that  $f (q) = 0$. It  follows  from Lemma  \ref{LeastActLemma}  that $f (p) \geq 0$ for all $p  \in \Gamma$.  We thus have 
\begin{align*}
 \E_q(D') &= \E_q(D) + 2 \langle D' -d(q), \Delta(f) \rangle + \E_q(\Delta(f)) \\
&> \E_q(D) + 2f (D) \\
&\geq  \E_q(D) \\
\end{align*}
as desired.
\end{proof}

\begin{corollary} \label{atMostCor}
 For every $D \in \Div(\Gamma)$, there is at most one $q$-reduced divisor equivalent to $D$. 
\end{corollary}
In order to deduce Theorem \ref{UniquenessThm} from Theorem \ref{EnergyThm}, we need some auxiliary results. 

\begin{proposition} \label{ExistProp}
Every divisor $D$  on $\Gamma$  is equivalent to a divisor which is effective outside $q$. 
\end{proposition}
\begin{proof}

This is proved in \cite{Hladkyetal} and \cite{MK08}.  Here is a simple variant of the proof in  \cite{Hladkyetal} (along the lines of a similar argument in \cite{BN1}).  Let $S$ be the set of branch points of $\Gamma$ together with $q$ and all points in the support of $D$.  If
$p  \in \Gamma\backslash \{q\}$, define  $\lambda_S (p)$  (the  ``level  of $p$'' relative  to  $S$)  to  be  the 
minimal  number  of elements of  $S\backslash \{q\}$ contained in a path from $q$ to $p$. We define $\lambda_S (q)$ to be $-1$.   Let $k$  be  the  maximum  value  of  $\lambda_S(p)$  for $p \in \Gamma$. For  $i= 0, \cdots, k$ let $S_i  = \{p  \in \Gamma \, | \, \lambda_S (p) = i\}$.  For $i = 0,\cdots, k -1$, define $A_i$ 
to be the closure of $\{p  \in \Gamma \, | \, \lambda_S (p) < i\}$ and define $B_i  = \{p  \in \Gamma \, | \, \lambda_S (p) > i\}$.  We have $\partial A_i  = S_i$ and $\partial B_i    = S_{i+1}$.  For $i = 0, \cdots , k-1$, let $f_i$   be the unique element of $R(\Gamma)$ which is $0$ on $A_i$ and $1$ on $B_i$ .  Define $P_i = \Delta(f_i)$.  By definition, 
\[ 
 P_i  = \sum_{v \in S_{i+1}} {b_v^{(i)} (v)} - \sum_{v\in S_i}{a_v^{(i)} (v)} 
\]
where  $b_v^{(i)}  =  \outdeg_{B_i}(v)  \geq  1$  and $a_v^{(i)}  =  \outdeg_{A_i}(v)  \geq 1$.  It  is  straightforward to  verify  that  we  can  choose  positive  integers  $c_0, c_1, \cdots, c_{k-1}$ inductively,  starting  from  $c_{k-1}$ and  working  backwards,  in  such a way  that $P := \sum_{i=0}^{k-1}{c_iP_i}$ satisfies $D + P  \in |D|_q$.
\end{proof}

Let $\Meas_{+}^d  (\Gamma)$ be the space of nonnegative Borel measures of total mass $d$ on $\Gamma$.  By Alaoglu's theorem, $\Meas_{+}^d  (\Gamma)$ is compact in the weak* topology. 

\begin{lemma}  \label{uniformConvLemma}
Let  $\Gamma$  be  a  metric  graph.
\begin{enumerate}
\item[(a)] If  $f_n$ converges  to $f$  pointwise  on  $\Gamma$  with 
$f_n , f \in \BDV(\Gamma)$, then $\Delta(f_n)$  converges to $\Delta(f)$  in the weak* topology on $\Meas^0(\Gamma)$. 
\item[(b)] Conversely, fix  $q  \in \Gamma$ and suppose that  $\Delta(f_n)$  converges to $\Delta(f)$ 
in the weak* topology on $\Meas^0(\Gamma)$, with $f_n , f   \in \BDV(\Gamma)$ and $f_n (q)  =  f (q)  =  0$  for  all $n$.  
Then $f_n$ converges uniformly to $f$ on $\Gamma$.
\end{enumerate}
\end{lemma}

\begin{proof}
(1) Since $\CPA(\Gamma)$ is dense in $\mathcal{C}(\Gamma)$, we just need to prove
that  $\int g \Delta (f_n) \rightarrow \int g \Delta (f)$     for every  $g  \in  \CPA(\Gamma)$.      But  this  is
equivalent   to   $\int f_n \Delta (g) \rightarrow \int f \Delta (g)$, which follows from pointwise convergence since $\Delta (g)$ is a discrete measure.

(2) This follows from the continuity of $j_q (x, y)$ together with the identities $f_n(x) = \int j_q (x, y)(\Delta f_n)(y)$ and 
$f(x) = \int j_q (x, y)(\Delta f)(y)$.
\end{proof}

Let $D$ be an effective divisor of degree $d$ on $\Gamma$.  We can topologize $|D|$ by thinking of it as a subset of $\Meas_{+}^d  (\Gamma)$ and giving it the subspace topology. 
	
\begin{theorem}  \label{compactThm}
The space $|D|$, with its weak* topology as subset of $\Meas_{+}^d  (\Gamma)$, is compact. 
\end{theorem}

\begin{proof}
Since $\Meas_{+}^d (\Gamma)$   is  a  compact Hausdorff space by Alaoglu's  theorem,  it  suffices  to  show  that  $|D|$ is  a  closed subspace. Suppose $\mu_n \in |D|$ converges to $\mu \in \Meas_{+}^d (\Gamma)$.  
We want to show that $\mu \in |D|$. 
Since each $\mu_n$  is a discrete measure of total mass $d$ with nonnegative integer masses, it follows 
easily from Alaoglu's theorem and the ``Portmanteau theorem'' (see \cite[Proof of Proposition 9.5]{BakerRumelyBook})
that $\mu$ is also a discrete measure. 
Therefore $\mu  = \mu_D  + \Delta (f)$  for  some  $f \in \CPA(\Gamma)$. Since $\mu_n    = \mu_D  + \Delta (f_n)$  with $f_n  \in R(\Gamma)$ and $\mu_n \rightarrow \mu$, it suffices to 
prove that if $f_n \in R(\Gamma)$, $f \in \BDV(\Gamma)$, and $\Delta(f_n)$ converges to $\Delta(f)$, then $f  \in R(\Gamma)$.  We may assume without loss of generality that $f_n (q) = f(q) = 0$ for all $n$, so by Lemma \ref{uniformConvLemma}(b) $f_n$ converges uniformly to $f$  on $\Gamma$. But then every directional derivative of $f_n$ 
 converges to the corresponding directional derivative of $f$. Since the slopes of $f_n$ are integers, it follows that the slopes of $f$  are integers as well. Thus $f \in R(\Gamma)$ and $\mu \in |D|$ as desired. \qed

Alternatively, one can prove this using the results of \cite{GathmannKerber} as follows. 
Gathmann and Kerber show that $|D|$ can be given the structure of a
compact polyhedral complex by thinking of it as a subset of the compact space $\Gamma^{(d)} = \Gamma^{d} / S_d$  
(here $S_d$ denotes the symmetric group on $d$ letters).  The  natural  map  from  $|D|$  (with  the  Gathmann-Kerber topology)  to $\Meas_{+}^d  (\Gamma)$  is  a  continuous  injection  whose  image  is $|D|$ with  its  weak* topology. As  the  continuous  image  of  a  compact  space, $|D|$  is  therefore compact  in  the  weak* topology. (Note  that  since  a continuous  bijection	between  compact  Hausdorff  spaces  is  automatically  a homeomorphism,  it follows  that  the  Gathmann-Kerber  and  weak* topologies on $|D|$  are  the same.)
\end{proof}

\begin{remark}
One can also deduce Theorem~\ref{compactThm} from the results in \cite{MK08}. 
\end{remark}

\medskip

We now give the promised proof of Theorem \ref{UniquenessThm}. 

\medskip

\begin{proof}[Proof of Theorem \ref{UniquenessThm}]
By Proposition \ref{ExistProp} and the definition of reduced divisors,  we  can  assume  without  loss  of  generality  that  $D$  is effective. We endow $|D|$ with the weak*  topology as above.  The function $\E_q: |D| \rightarrow \RR$ is continuous since 
\[\E_q(D') = \int \int j_q(x, y) d\mu_{D'-d(q)}(x) d\mu_{D'-d(q)}(y)  \]
and $j_q (x, y)$ is jointly continuous in $x$ and $y$.  Since $|D|$ is compact,
$\E_q$  attains	its minimum value at some effective divisor $D_q \in |D|$. By
Theorem \ref{EnergyThm}, $D_q$ is $q$-reduced.  This proves the existence of a  $q$-reduced divisor
equivalent to $D$.  The uniqueness is Corollary \ref{atMostCor}. 
\end{proof}

\begin{remark}
The continuity of $\E_q: |D| \rightarrow \RR$ played an important role in the Proof of Theorem \ref{UniquenessThm}. 
More generally, the energy pairing is continuous as a function from $\Meas^0(\Gamma) \times \Meas^0(\Gamma)$ to $\RR$.  If we fix a reference point $q \in \Gamma$, this follows from the joint continuity of $j_q (x, y)$ as a function of $x$ and $y$ together with the formula
\[
\langle \mu, \nu \rangle= \int \int j_q (x, y)d\mu(x)d\nu(y) \ . 
\]
\end{remark}

\begin{remark}
The continuity of the energy pairing, together with the principle of least action (Lemma~\ref{LeastActLemma}), can be used to give a new proof of Amini's theorem \cite{Amini} that for $D$ effective, the map from $\Gamma$ to $|D|$ sending a point $q \in \Gamma$ to the unique $q$-reduced divisor $D_q$ equivalent to $D$ is continuous.

\end{remark}

We conclude this section with an analogue of Theorem~\ref{energyThm2} for metric graphs.
For a divisor $D\in \Div(\Gamma)$ of degree $d$, define its ``total potential'' $b_q(D)$ by
\begin{equation} \label{eq:metricbq}
b_q(D):=\int\int{j_q(x,y) d\mu_{D-d(q)}(x) dy} \ ,
\end{equation}
where $dy$ denotes the measure on $\Gamma$ whose restriction to each edge is Lebesgue measure. 

\begin{theorem} \label{EnergyThm2}
Let $D$  be  a  divisor  on $\Gamma$  which  is  effective outside $q$. Then $D$  is $q$-reduced if and only if 
$b_q(D) < b_q(D')$ for every divisor $D' \neq D$ in $|D|_q$. 
\end{theorem}

The proof is similar to the proofs of Theorem~\ref{energyThm2} and Theorem \ref{EnergyThm}, so we omit it.


\subsection{Dhar's algorithm for metric graphs and Luo's theorem}

In this section, we briefly discuss some results from Ye Luo's paper \cite{Luoye} and their relation to the material in the
present paper.

\medskip

The idea behind Dhar's algorithm for metric graphs \cite[Algorithm 2.5]{Luoye} is the same as for non-metric graphs.
Let $\Gamma$ be a metric graph, fix a point $q \in \Gamma$, and let $D$ be a divisor on $\Gamma$ which is
effective outside $q$.
Think of the metric graph $\Gamma$ as being made of a flammable material.
At every point $p \in \Gamma$ with $p \neq q$, there are $D(p)$ firefighters, each of whom can control fires in a single direction leading into $p$.
A fire starts at $q$ and proceeds along each direction emanating from $q$; whenever there are fires approaching $p$ in more than $D(p)$ directions, the fire burns through $p$ and 
proceeds to burn in all directions emanating from $p$.  (In particular, if $e$ is a segment in $\Gamma$
consisting entirely of points $p$ with $D(p)=0$, then a fire starting at one endpoint of $e$ will burn through $e$
unobstructed.)  The divisor $D$ is $q$-reduced iff the fire eventually burns through all of $\Gamma$.

\medskip

For a formal statement of Dhar's algorithm for metric graphs and a proof of correctness, see \cite[Algorithm 2.5 and Lemma 2.6]{Luoye}.

\medskip

In order to compute the unique $q$-reduced divisor equivalent to a given divisor $D$, 
one can proceed along the lines of Algorithm~\ref{MainAlgorithm} above.  
Starting with a divisor $D$, the first step is to replace $D$ with an equivalent divisor which is effective
outside $q$.  This is accomplished by following the proof of Proposition~\ref{ExistProp}.
The second step is to iterate Dhar's algorithm until we reach a $q$-reduced divisor.  More specifically,
if $D$ is not yet reduced then by running Dhar's algorithm on $D$ we obtain a proper connected subset $Y \neq \Gamma$ of 
$\Gamma$ which contains $q$ (the ``burned portion'' of $\Gamma$).  Choose a connected component $X$ of the
complement of $Y$ and (as in the proof of Lemma~\ref{ReducedDefLem}) 
choose $\epsilon>0$ maximal with respect to the property that there is a rational 
function $f \in R(\Gamma)$ which is $0$ on $X$ and $\epsilon$ outside of an $\epsilon$-neighborhood of $X$, with
constant slope $1$ in between.  By construction, the divisor $D+\Delta(f)$ is still effective outside $q$.
Replacing $D$ by $D+\Delta(f)$ and iterating this procedure, one obtains (after finitely many iterations)
the unique $q$-reduced divisor equivalent to $D$.

\medskip

For a formal statement of this algorithm and a proof that it terminates after a finite number of iterations, 
see \cite[Algorithm 2.12 and Theorem 2.14]{Luoye}.

\medskip

In \eqref{eq:metricbq} we defined a metric graph analogue of the functional $b_q$ from \S\ref{ChipPotSec}.  
As in Proposition~\ref{energyproposition2}(b), this functional is a ``monovariant'' which decreases by an explicit function of $\epsilon$ when $D$ is replaced by $D + \Delta(f)$ (with $f$ as above):
\[
\begin{aligned}
\ b_q(D+\Delta(f))-b_q(D) &=\int \int j_q(x,y) \Delta_x(f) dy\\
&=\int \int f(x)\Delta_x(j_q(x,y)) dy\\
&=\int \int f(x) (\delta_y(x)-\delta_q(x)) dy\\
&=\int (f(y)-f(q)) dy\\ 
&=-(l(X)\epsilon + \frac{\lambda(X)}{2} \epsilon^2) \ ,
\end{aligned}
\]
where $l(X)$ is the total length of $X$ and $\lambda(X)$ is the size of the $(X,\Gamma\backslash X)$-cut.

It would be interesting to give explicit lower bounds for the $\epsilon$'s which can appear and thus obtain a running-time
analysis of Luo's Algorithm 2.12 similar to our Proposition~\ref{time}(c); we leave this as an open problem for future research.



\bibliographystyle{plain}
\bibliography{ChipPot11}

\end{document}